\documentclass[11pt]{amsart}
\usepackage{tabularx,booktabs}
\usepackage{caption}
\usepackage{amsmath,tikz}
\usepackage{amsfonts,mathscinet}
\DeclareMathOperator{\arcsinh}{arcsinh}
\usepackage{amscd}
\usepackage{amsthm}
\usepackage{amssymb} \usepackage{latexsym}
\usepackage{eufrak}
\usepackage{euscript}
\usepackage{epsfig}
\usepackage{graphics}
\usepackage{array}
\usepackage{enumerate}
\usepackage{dsfont}
\usepackage{color}
\usepackage{wasysym}
\usepackage{hyperref}
\usepackage{pdfsync}
\usepackage{bm}

\newcommand{\bel}[1]{\begin{equation}\label{#1}}

\newcommand{\be}{\begin{equation}}

\newcommand{\ba}{\begin{eqnarray}}
\newcommand{\ea}{\end{eqnarray}}

\newcommand{\qe}{\end{equation}}
\newcommand{\R}{{\mathds{R}}}

\newcommand{\wt}{\widetilde}

\newcommand{\MT}{\mathcal{M}_{\mathcal{T}}}

\newcommand{\vol}{\mathrm{vol}}

\newcommand{\br}{{\pmb{r}}}

\newcommand{\EE}{\mathds{E}}
\newcommand{\HH}{\mathds{H}}

\newcommand{\Hmm}[1]{\leavevmode{\marginpar{\tiny%
$\hbox to 0mm{\hspace*{-0.5mm}$\leftarrow$\hss}%
\vcenter{\vrule depth 0.1mm height 0.1mm width \the\marginparwidth}%
\hbox to
0mm{\hss$\rightarrow$\hspace*{-0.5mm}}$\\\relax\raggedright #1}}}

\newtheorem{theorem}{Theorem}[section]

\newtheorem{lemma}[theorem]{Lemma}
\newtheorem{corollary}[theorem]{Corollary}
\newtheorem{definition}[theorem]{Definition}

\newtheorem{remark}[theorem]{Remark}

\newtheorem{proposition}[theorem]{Proposition}
\newtheorem{problem}[theorem]{Problem}

\begin{document}

\title[3-dimensional hyperbolic combinatorial Yamabe flow]{3-dimensional combinatorial Yamabe flow in hyperbolic background geometry}

\author{Huabin Ge}
\email{hbge@bjtu.edu.cn}
\address{Huabin Ge: Department of Mathematics, Beijing Jiaotong University, Beijing 100044, P.R. China}

\author{Bobo Hua}
\email{bobohua@fudan.edu.cn}
\address{Bobo Hua: School of Mathematical Sciences, LMNS,
Fudan University, Shanghai 200433, China; Shanghai Center for
Mathematical Sciences, Fudan University, Shanghai 200433,
China.}

\begin{abstract}
We study the 3-dimensional combinatorial Yamabe flow in hyperbolic background geometry. For a triangulation of a 3-manifold, we prove that if the number of tetrahedra incident to each vertex is at least 23, then there exist real or virtual ball packings with vanishing (extended) combinatorial scalar curvature, i.e. the (extended) solid angle at each vertex is equal to $4\pi$. In this case, if such a ball packing is real, then the (extended) combinatorial Yamabe flow converges exponentially fast to that ball packing. Moreover, we prove that there is no real or virtual ball packing with vanishing (extended) combinatorial scaler curvature if the number of tetrahedra incident to each vertex is at most 22.

  \bigskip


 

 \end{abstract}

\maketitle
\tableofcontents



\par
\maketitle

\bigskip





\section{Introduction}\label{sec:intro}


In the seminar work \cite{T1}, Thurston introduced (2-dimensional) circle packings to construct hyperbolic 3-manifolds or 3-orbifolds. To extend Thurston's circle packings to higher dimension, Cooper and Rivin \cite{CR} studied the deformation of ball packings, which are three dimensional analogs of circle packings. Inspired by the pioneering work of Chow and Luo \cite{CL1}, where the combinatorial surface Ricci flows based on circle packings were introduced, Glickenstein \cite{G1,G2} introduced a combinatorial version of Yamabe flow based on Euclidean triangulations defined by ball packings (Luo \cite{L1} also introduced a combinatorial Yamabe flow on surfaces). The first author of the paper, Jiang and Shen \cite{GJS} studied the convergence of Glickenstein's Yamabe flow for regular ball packings in Euclidean background geometry. In this paper, we introduce the combinatorial Yamabe flow for triangulations of 3-manifolds in hyperbolic background geometry and study the convergence of the flow.

Let $M$ be a closed 3-manifold with a triangulation $\mathcal{T}=\{\mathcal{T}_0,\mathcal{T}_1,\mathcal{T}_2,\mathcal{T}_3\}$, where the symbols $\mathcal{T}_0,\mathcal{T}_1,\mathcal{T}_2,\mathcal{T}_3$ represent the sets of vertices, edges, faces and tetrahedra respectively. The pair $(M, \mathcal{T})$ refers to a 3-manifold with a triangulation $\mathcal{T}.$ For simplicity, we write $\mathcal{T}_0=\{1,\cdots,N\}$, where $N$ is the number of vertices, and denote by $\{ij\}$ ($\{ijk\}$ and $\{ijkl\}$ resp.) an edge (a triangle and a tetrahedron resp.) in the triangulation.  A tetrahedron is called incident to a vertex if the latter is a vertex of the former.  Let $\mathds{R}^3$ ($\mathds{H}^3$ resp.) be the simply-connected 3-dimensional Euclidean space (hyperbolic space of constant sectional curvature $-1$ resp.). For a triangulated 3-manifold $(M,\mathcal{T})$, Cooper and Rivin \cite{CR} constructed a piecewise linear metric (hyperbolic metric resp.) by ball packings: assign a positive number $r_i$ to each vertex $i$ which serves as the radius of a  ball in $\R^3$ ($\mathds{H}^3$ resp.) centered at $i,$ construct a Euclidean (hyperbolic resp.) tetrahedron in $\R^3$ ($\mathds{H}^3$ resp.) of edge length $l_{ij}=r_{i}+r_{j}$ for each tetrahedron $\{ijkl\}\in \mathcal{T}_3$ if it exits, and glue them together along common faces. A Euclidean or hyperbolic ball packing, also called sphere packing, can be completely described as a map $$\br=(r_1,\cdots,r_N):\mathcal{T}_0\rightarrow \R_+:=(0,+\infty).$$ In the following, we always write bold symbols for vectors in Euclidean space such as $\br$ above. Geometrically, a Euclidean (hyperbolic resp.) ball packing $\br$ attaches to each vertex $i\in\mathcal{T}_0$ a Euclidean (hyperbolic resp.) ball $B_i$ of radius $r_i$ centered at $i$ such that two balls $B_i$ and $B_j$ are externally tangent for any $\{ij\}\in \mathcal{T}_1$. Note that for each $\{ijkl\}\in \mathcal{T}_3,$ the assignment $(r_i,r_j,r_k,r_l)$ can form a non-degenerate Euclidean tetrahedron if and only if, see e.g. \cite{G1},
\begin{equation*}\label{nondegeneracy condition}
Q^\EE(r_i,r_j,r_k,r_l):=\left(\frac{1}{r_{i}}+\frac{1}{r_{j}}+\frac{1}{r_{k}}+\frac{1}{r_{l}}\right)^2-
2\left(\frac{1}{r_{i}^2}+\frac{1}{r_{j}^2}+\frac{1}{r_{k}^2}+\frac{1}{r_{l}^2}\right)>0.
\end{equation*}
The case $Q^\EE(r_i,r_j,r_k,r_l)=0$ corresponds to Descartes' circle theorem, also called Soddy-Gossett theorem, see \cite{CR}. Analogously, $(r_i,r_j,r_k,r_l)$ can form a non-degenerate hyperbolic tetrahedron if and only if, see e.g. \cite{Lag-MW,Mau},
\begin{eqnarray}\label{hyper-tetra}&&\\
Q^{\mathds{H}}(r_i,r_j,r_k,r_l)&:=& \left(\coth{r_i}+\coth{r_j}+\coth{r_k}+\coth{r_l}\right)^2-\nonumber\\&&
2\left(\coth{r_i}^2+\coth{r_j}^2+\coth{r_k}^2+\coth{r_l}^2\right)+4>0.\nonumber
\end{eqnarray} A Euclidean (hyperbolic resp.) ball packing $\br$ is called \emph{real} if all the Euclidean (hyperbolic resp.) tetrahedra are non-degenerate, and called \emph{virtual} otherwise. Cooper and Rivin \cite{CR} introduced the set of ball packings of a triangulation as a discrete analogue to the conformal class of a Riemannian metric on a manifold.

In this paper, we consider hyperbolic ball packings. In the following, we always refer to the hyperbolic background geometry, and for the sake of simplicity call a hyperbolic ball packing a ball packing, etc..  For a triangulated 3-manifold $(M,\mathcal{T}),$ we denote by $\mathcal{M}_{\mathcal{T}}$ the set of real ball packings. $\MT$ is a simply connected open subset of $\mathds{R}^N_{+}$, which was first obtained by Cooper and Rivin \cite{CR} (or see the first paragraph of Section \ref{subsec:g-flow}). The combinatorial scalar curvature on a real ball packing was introduced by Cooper and Rivin \cite{CR}. For $\br\in \MT,$ we denote by $\alpha_{ijkl}$ the solid angle at the vertex $i$ in the hyperbolic tetrahedron $\{ijkl\}\in \mathcal{T}_3.$ The combinatorial scalar curvature at the vertex $i$ is defined as
\begin{equation}\label{Def-CR curvature}
K_{i}:= 4\pi-\sum_{\{ijkl\}\in \mathcal{T}_3}\alpha_{ijkl},
\end{equation}
where the summation is taken over all tetrahedra incident to $i.$ It is remarkable that
the combinatorial curvature map
\begin{eqnarray*}\pmb{ K}:&&\mathcal{M}_{\mathcal{T}}\to \mathds{R}^N,\\ &&\br\mapsto \pmb{K}(\br)=(K_1,K_2,\cdots, K_N)\end{eqnarray*} is locally injective by Cooper and Rivin \cite{CR} and globally injective by Xu \cite{Xu}, which generalize Thurston's rigidity about surface circle packings \cite{T1}.

In the Euclidean setting, Glickenstein \cite{G1} proposed to study the combinatorial Yamabe problem, i.e. the existence of constant combinatorial scalar curvature in the conformal class $\MT.$
To approach the problem, he introduced a combinatorial Yamabe flow
\begin{equation}\label{Def-Flow-Glickenstein}
\frac{dr_i}{dt}=-K_ir_i, \quad \forall i\in \mathcal{T}_0,
\end{equation}
aiming to deform the initial ball packing to the one with constant combinatorial scalar curvature. The prototype of (\ref{Def-Flow-Glickenstein}) is Chow and Luo's combinatorial Ricci flow \cite{CL1} and Luo's combinatorial Yamabe flow \cite{L1} on surfaces, see also \cite{Ge}-\cite{GZX}.

Inspired by Glickenstein's problem, we propose the following combinatorial Yamabe problem in hyperbolic setting.
\begin{problem}[Hyperbolic combinatorial Yamabe problem]\label{prob:1} Given a triangulated 3-manifold $(M,\mathcal{T}),$ does there exist a hyperbolic ball packing with vanishing combinatorial scaler curvature? \end{problem}

To solve the problem, one possible approach, initiated in the geometric analysis, is to use the variational method. We introduce the Cooper-Rivin's functional $\mathcal{S}$ on real ball packings, \begin{eqnarray}\label{eq:cprivinf}{\mathcal{S}}(\br)=\sum_{i\in \mathcal{T}_0}{K}_i r_i-2\vol_M(\br),\forall \ \br\in \MT,
\end{eqnarray} where $\vol_M(\br):=\sum_{\{ijkl\}\in \mathcal{T}_3}\vol_{ijkl}(\br)$ denotes the total volume of the triangulation in ball packing metric $\br.$ The hyperbolic combinatorial Yamabe invariant of $(M,\mathcal{T})$ is defined as
$$Y_\mathcal{T}:=\inf_{\br\in \MT} {\mathcal{S}}(\br).$$ We prove that the solutions to hyperbolic combinatorial Yamabe problem are exactly the critical points of the functional ${\mathcal{S}}$.
\begin{theorem}\label{thm:Yamabeequivalent} The following are equivalent:
\begin{enumerate}
\item There exists a real ball packing $\overline{\br}$ such that $\pmb{K}(\overline{\br})=0,$ i.e. the combinatorial Yamabe problem is solvable.
\item The functional $\mathcal{S}$ has a critical point in $\MT.$
\item The hyperbolic combinatorial Yamabe invariant $Y_\mathcal{T}$ is attained by a real ball packing, i.e. the functional $\mathcal{S}$ has a global minimizer.
\end{enumerate}
\end{theorem}

In the paper, we mainly follow Glickenstein's parabolic method to attack the combinatorial Yamabe problem. We introduce the hyperbolic combinatorial Yamabe flow as follows. We say that $\br(t)$ is a solution to the hyperbolic combinatorial Yamabe flow with the initial data $\br(0),$ if for all $t\in [0,T),$ $\br(t)\in \MT$ and
\begin{equation}\label{hyperyflow}
\frac{dr_i}{dt}=-K_i\sinh{r_i},\quad \forall i\in \mathcal{T}_0.
\end{equation} The stationary points of the flow are given by real packings $\overline{\br}$ with $\pmb{K}(\overline{\br})=0.$
This makes good sense to use the flow \eqref{hyperyflow} to find ball packings with vanishing combinatorial scalar curvature. Since the right hand side of \eqref{hyperyflow} is locally Lipschitz in $\MT,$ by the ODE theory the solution exists and is unique in the maximal existence time $T\leq \infty$ for any initial data $\br(0)\in \MT.$ We say that the flow \eqref{hyperyflow} converges if the solution $\br(t)$ exists for $[0,\infty)$ and there is $\overline{\br}\in \MT$ such that
\begin{equation}\label{eq:converges1}\br(t)\to \overline{\br},\quad t\to\infty.\end{equation} It is well-known that if the flow \eqref{hyperyflow} converges to $\overline{\br},$ then $$\pmb{K}(\overline{\br})=0,$$ see Proposition~\ref{prop-converg-imply-const-exist} below. However, the convergence of the flow \eqref{hyperyflow} is in general difficult to obtain. We have the following scenarios:
\begin{enumerate}[(i)]
\item The flow \eqref{hyperyflow} may develop finite time singularity, i.e. the maximal existence time $T<\infty.$
\item A triangulated 3-manifold $(M,\mathcal{T})$ may not admit a real ball packing of vanishing combinatorial scaler curvature, see Theorem~\ref{thm:less22} below, which implies that the flow never converges for any initial data.
\end{enumerate}

In order to circumvent the difficulties, we introduce the extended combinatorial Yamabe flow. First, for any tetrahedron $\{ijkl\}\in \mathcal{T}_3$ we may continuously extend solid angles $\alpha_{ijkl}$ for real ball packings to $\wt{\alpha}_{ijkl}$ for all ball packings including virtual ones. This induces the definition of extended combinatorial scaler curvature for any ball packing,
\begin{equation}\label{extended curvature}
\wt{K}_{i}:= 4\pi-\sum_{\{ijkl\}\in \mathcal{T}_3}\wt{\alpha}_{ijkl}.
\end{equation}
Then the extended combinatorial Yamabe flow is defined as follows: $\br(t)\in \R^N_+$ for $t\in [0,T)$ and
 \begin{equation}\label{extflow}
\frac{dr_i}{dt}=-\wt{K}_i\sinh{r_i},\quad \forall i\in \mathcal{T}_0.
\end{equation}
We prove the long time existence of the extended combinatorial Yamabe flow, i.e. the maximal existence time $T=\infty,$ in Theorem~\ref{thm-yang-write}, and the uniqueness of the solution to the flow in Theorem~\ref{thm:uniqueness}.
We say that the flow \eqref{extflow} converges to $\overline{\br}\in \R^N_+$ if $$\br(t)\to \overline{\br},\quad t\to\infty.$$

With the help of the extended combinatorial Yamabe flow and hyperbolic geometry, we prove the existence result of hyperbolic Yamabe problem in the extended sense under some combinatorial condition. For any vertex $i,$ we denote by $d_i$ the number of tetrahedra incident to $i$ and call it the tetra-degree of $i.$
\begin{theorem}\label{thm:23degree} Let $(M,\mathcal{T})$ be a triangulated 3-manifold satisfying
$$d_i\geq 23,\quad \forall i\in \mathcal{T}_0.$$ Then there exists a ball packing $\br,$ real or virtual, such that \begin{equation}\label{eq:eq0-1}\pmb{\wt{K}}(\br)=0.\end{equation} If there is a real ball packing $\br$ satisfying \eqref{eq:eq0-1}, then the extended flow \eqref{extflow} converges exponentially fast to $\br$ for any initial data $\br(0)\in \R^N_+.$
\end{theorem}
To prove the existence of ball packings with vanishing extended combinatorial scaler curvature, we consider the solution $\br(t)$ to extended flow \eqref{extflow}. If the tetra-degree at each vertex is at least 23, then we obtain uniform lower and upper estimates of the solution $\br(t)$ in Theorem~\ref{thm:lowerest} and Theorem~\ref{thm:upperest}. This yields that there is a sequence $t_n\to \infty$ such that $$\pmb{\wt{K}}(\br(t_n))\to \pmb{0}.$$ By further extracting a subsequence, we obtain a limiting ball packing $\br$ with vanishing extended combinatorial scaler curvature.
\begin{remark}
\begin{enumerate}
\item If there is a real ball packing with vanishing combinatorial scaler curvature, then there is no virtual ones with the same property. In this case, the real ball packing with vanishing combinatorial scaler curvature is unique, see Theorem~\ref{thm:alternative}.

 \item We will prove that there is no real or virtual ball packing with vanishing extended combinatorial scaler curvature if the tetra-degree at each vertex is bounded above by 22, see Theorem~\ref{thm:less22}. This indicates the tetra-degree assumption is essential.
 \end{enumerate}
\end{remark}

In the next theorem, we prove the non-existence result of hyperbolic Yamabe problem in the extended sense if the tetra-degrees at vertices are small.
\begin{theorem}\label{thm:less22} Let $(M,\mathcal{T})$ be a triangulated 3-manifold satisfying
$$d_i\leq 22,\quad \forall i\in \mathcal{T}_0.$$ There exists no ball packing $\br,$ real or virtual, such that \begin{equation}\label{eq:eq0-10}\pmb{\wt{K}}(\br)=0.\end{equation} Moreover, for any initial data $\br(0)\in \R^N_+,$ the solution $\br(t)$ of extended combinatorial Yamabe flow \eqref{extflow} satisfies
 $$\br(t)\to \pmb 0,\quad t\to \infty.$$
\end{theorem}

A triangulation $(M,\mathcal{T})$ is called tetra-regular if all tetra-degrees of vertices are equal. As a corollary of above results, we have complete results for tetra-regular triangulations.
\begin{corollary}\label{coro:regulartetra} For a tetra-regular triangulation of 3-manifold $(M,\mathcal{T})$ with tetra-degree $d,$ we have the following:
\begin{enumerate}
 \item For $d\geq 23,$ there exists a unique ball packing with vanishing combinatorial scaler curvature, and the solution to the extended flow \eqref{extflow} converges exponentially fast to that ball packing for any initial data.
 \item For $d\leq 22,$ there is no real or virtual ball packing with vanishing extended combinatorial scaler curvature, and the solution to the extended flow \eqref{extflow} converges to zero for any initial data.
 \end{enumerate}

\end{corollary}

The paper is organized as follows:
In the next section, we study the geometry of hyperbolic tetrahedra, the extension of solid angles to virtual tetrahedra, and some comparison principles for solid angles. In Section~\ref{s:ballpackings}, we study global properties of triangulations of 3-manifolds configured by ball packings, and prove Theorem~\ref{thm:Yamabeequivalent}, the long time existence and the uniqueness of the extended combinatorial Yamabe flow. In the last section, we prove the main results, Theorem~\ref{thm:23degree}, Theorem~\ref{thm:less22} and Corollary~\ref{coro:regulartetra}.
\section{Geometry of hyperbolic tetrahedra}
\subsection{Hyperbolic tetrahedra configured by four tangent balls}
\label{section-extend-K}
In this section, we consider the geometry of a hyperbolic tetrahedron. For any $\{v_1,v_2,v_3,v_4\}\subset \HH^3,$ we denote by $\tau_{v_1v_2v_3v_4}$ the tetrahedron in $\HH^3$ with vertices $\{v_i\}_{i=1}^4,$ i.e. the convex hull of  $\{v_i\}_{i=1}^4$ in $\HH^3.$ We call the tetrahedron $\tau_{v_1v_2v_3v_4}$ a \emph{real}, or \emph{non-degenerate}, tetrahedron if the four vertices are  in general position, i.e. none of them is contained in the convex hull of the other three. \begin{definition}\label{tetra-def} For a real hyperbolic tetrahedron $\tau_{v_1v_2v_3v_4},$ we denote by $l_{ij}$ the edge length of the edge $v_iv_j,$ by $\beta_{ij}$ the dihedral angle at the edge $v_iv_j,$ by $\alpha_i$ the solid angle at the vertex $v_i,$ and by $\vol_{v_1v_2v_3v_4},$ $\vol$ for short, the volume of $\tau_{v_1v_2v_3v_4}.$
\end{definition} It is well-known that the tetrahedron is determined, up to an isometry of $\HH^3$, by the data $\pmb{l}=(l_{12},l_{13},l_{14},l_{23},l_{24},l_{34})$ or $\pmb{\beta}=(\beta_{12},\beta_{13},\beta_{14},\beta_{23},\beta_{24},\beta_{34}).$ Hence $\pmb{\beta}(\pmb{l})$ is a bijective map for a tetrahedron, see e.g. \cite{Vinberg} for more details. The Gram matrix in terms of edge lengths is given by
\[G:=\left(\begin{matrix}
  -1 & -\cosh l_{12} & -\cosh l_{13} &-\cosh l_{14}  \\
  -\cosh l_{12} & -1 & -\cosh l_{23} &-\cosh l_{24}\\
  -\cosh l_{13} & -\cosh l_{23} & -1 &-\cosh l_{34} \\
  -\cosh l_{14} & -\cosh l_{24} & -\cosh l_{34} &-1 \\
  \end{matrix}\right).\]
In this paper, we adopt the notation $\{i,j,k,l\}$ for a rearrangement of $\{1,2,3,4\},$ i.e. $i,j,k,l$ are distinct.
It is also well-known, see e.g. \cite{MU05}, that for any $1\leq i<j\leq 4,$
\begin{equation}\label{eq:eq2-4}\cos\beta_{ij}=\frac{c_{kl}}{\sqrt{c_{kk}c_{ll}}},\end{equation} where $c_{kl}$ is the $(k,l)$-cofactor of the matrix $G,$ i.e. $c_{kl}=(-1)^{k+l}\det G_{kl},$ where $G_{kl}$ is the matrix obtained from $G$ by deleting $k$-th row and $l$-column.

The Schl\"{a}fli formula, see e.g. \cite{Vinberg}, reads as
$$-2d \vol =\sum_{1\leq i<j\leq 4}l_{ij}d\beta_{ij}.$$ Note that the Jacobian matrix
$\frac{\partial\pmb{ \beta}}{\partial \pmb l}$ is symmetric and non-degenerate.

We denote by $\tau=\{1234\}$ the combinatorial tetrahedron of vertices $\{1,2,3,4\}$ which contains only combinatorial information. For the convenience, we write $\tau=\{1234\}=\{ijkl\}$ where $\{i,j,k,l\}$ is a rearrangement of $\{1,2,3,4\}$. For any $\br=(r_1,r_2,r_3,r_4)\in\mathds{R}^4_+$, we write
$$Q(\br):=Q^{\HH}(r_1,r_2,r_3,r_4),$$ where $Q^{\HH}(\cdot)$ is defined in \eqref{hyper-tetra} in the introduction. We would like to endow $\tau$ with a metric structure isometric to a hyperbolic tetrahedron $\tau(\br)$ in $\HH^3$ such that the edge length of each edge $\{ij\}$ is given by \begin{equation}\label{eq:eq2-2}l_{ij}=r_i+r_j,\end{equation} which is unique up to an isometry in $\HH^3$ if it exists. It is well-known, see \cite{Lag-MW,Mau}, that the tetrahedron $\tau(\br)$ exists and is non-degenerate if and only if $Q(\br)>0.$ In this case, we call $\tau=\{1234\}$ with the ball packing $\br,$ denoted by $\tau(\br),$ a \emph{real tetrahedron}.
Otherwise, for $Q(\br)\leq 0,$ $\tau(\br)$ is degenerate and we call it a \emph{virtual tetrahedron}.
The set of real tetrahedra on $\tau=\{1234\}$ is given by
\begin{equation}
\Omega_{1234}=\left\{\br=(r_1,r_2,r_3,r_4)\in\R^4_{+}:Q(\br)>0\right\}.
\end{equation}
It is simply-connected and its boundary is piecewise analytic. In fact, one can prove that $$\coth(\Omega_{1234})=\{(\coth r_1,\cdots,\coth r_4):(r_1,r_2,r_3,r_4)\in\Omega_{1234}\}$$ is a star-shaped domain with respect to the point $(\coth1,\coth1,\coth1,\coth1),$ and hence it is simply-connected. This yields $\Omega_{1234}$ is simply-connected by the homeomorphism between two sets.

In this paper, we denote by $\pmb{1}$ the vector in Euclidean space whose entries are constant $1.$ It is obvious that $t\pmb{1}\in \Omega_{1234}$ for any $t>0.$
For a real tetrahedron $\tau(\br)$, we denote by $\beta_{ij}(\br),$ $\alpha_i(\br)$ and $\vol(\br)$ the corresponding quantities in Definition~\ref{tetra-def}, and omit the dependence of $\br$ if it is clear in the context.
Note that for any vertex $i,$
\begin{equation}\label{eq:eq2-3}\alpha_i=\beta_{ij}+\beta_{ik}+\beta_{il}-\pi.\end{equation} The Schl\"{a}fli formula implies that
$$-2d \vol =\sum_{i=1}^4r_id\alpha_i.$$
Cooper and Rivin \cite{CR} introduced the following functional on the set of real tetrahedra,
\begin{eqnarray}\label{eq:crivin}\mathcal{U}: &&\Omega_{1234} \to \R, \\&&\br\mapsto \mathcal{U}(\br)=\sum_{i=1}^4\alpha_ir_i+2\vol.\nonumber\end{eqnarray} By the Schl\"{a}fli formula,
\begin{equation}\label{eq:1-1}d \mathcal{U}(\br)=\sum_{i=1}^4\alpha_id r_i.\end{equation} This implies that
\begin{eqnarray}&&\nabla \mathcal{U}(\br)=\pmb{ \alpha}\label{eq:firstder}\\&&
\nabla^2 \mathcal{U}(\br)=\frac{\partial\pmb{ \alpha}}{\partial \br}\label{eq:secondder},\end{eqnarray} where $\frac{\partial\pmb{ \alpha}}{\partial \br}$ denotes the Jacobian matrix of the map $\br\mapsto \pmb{\alpha}(\br).$
\begin{lemma}[Rivin \cite{Ri}]\label{lem:rivin}
For any $\br\in \Omega_{1234},$ $\frac{\partial\pmb{ \alpha}}{\partial \br}$ is a positive definite matrix.
\end{lemma}
\begin{proof} For the convenience of the readers, we recall the proof here which was given by Rivin in \cite{Ri}.
We denote by $\pmb{l}(\br)$ and $\pmb{\alpha}(\pmb{\beta})$ the maps defined in \eqref{eq:eq2-2} and \eqref{eq:eq2-3} respectively. Then one has $\pmb\alpha(\br)=\pmb\alpha\circ\pmb\beta\circ\pmb{l}(\br).$ This yields that
$$\frac{\partial\pmb{ \alpha}}{\partial \br}=\frac{\partial\pmb{\alpha}}{\partial \pmb\beta}\frac{\partial\pmb{\beta}}{\partial \pmb{l}}\frac{\partial \pmb{l}}{\partial \br}.$$ By the calculation,
\[
  \frac{\partial\pmb{\alpha}}{\partial \pmb\beta}=\left(\frac{\partial \pmb{l}}{\partial \br}\right)^T=
  \left(\begin{matrix}
  1 & 1 & 1 &0 & 0& 0 \\
  1 & 0 & 0 &1 & 1& 0 \\
  0 & 1 & 0 &1 & 0& 1 \\
  0 & 0 & 1 &0 & 1& 1 \\
  \end{matrix}\right).
\] Note that the eigenvalues of $\frac{\partial\pmb{ \alpha}}{\partial \br}$ and those of
$\frac{\partial\pmb{\beta}}{\partial \pmb{l}}\frac{\partial \pmb{l}}{\partial \br}\frac{\partial\pmb{\alpha}}{\partial \pmb\beta}$ are same up to some zeros. Since $\frac{\partial\pmb{\beta}}{\partial \pmb{l}}$ is non-degenerate and $\frac{\partial \pmb{l}}{\partial \br}\frac{\partial\pmb{\alpha}}{\partial \pmb\beta}$ has four nonzero eigenvalues, $\frac{\partial\pmb{ \alpha}}{\partial \br}$ is non-degenerate for any $\br\in \MT$ by considering the ranks of matrices. By the computation at a specific point $\br={\pmb{1}}=(1,1,1,1),$ one can show that $\frac{\partial\pmb{ \alpha}}{\partial \br}(\pmb{1})$ is positive definite, see \eqref{eq:eqa-1} in Appendix (or see the Appendix in \cite{GZX}, where $\frac{\partial\pmb{\beta}}{\partial \pmb{l}}$ is calculated for regular tetrahedra.). Hence by the continuity, $\frac{\partial\pmb{ \alpha}}{\partial \br}$ is positive definite for any $\br\in \Omega_{1234}.$

\end{proof}

To abbreviate the notation in the paper, we introduce the following change of variables.
\begin{definition}\label{def:rtoy} For any $\br=(r_1,r_2,r_3,r_4)\in \R^4_+,$ we define
$$y_i:=\coth r_i,\quad \forall\ 1\leq i\leq 4.$$
\end{definition}
By \eqref{eq:eq2-4}, we have an explicit formula for the dihedral angles of a real tetrahedron.
For a real tetrahedron $\tau(\br),$
\begin{eqnarray}\label{eq:eq2-5}
\cos\beta_{ij}(\br)&=&\frac{\sinh r_i\sinh r_j\sqrt{\sinh r_k\sinh r_l}}{4\sqrt{\sinh(r_i+r_j+r_k)\sinh(r_i+r_j+r_l)}}\times\\&&(Q(\br)-(y_i+y_j)^2+(y_k-y_l)^2).\nonumber
\end{eqnarray}

By taking the inverse function, $\beta_{ij}=\arccos(\cdots),$ from which we can calculate
$$\frac{\partial \beta_{ij}}{\partial r_i}, \frac{\partial \beta_{ij}}{\partial r_k},\cdots,$$ see the formulae in Appendix.
Using these results and \eqref{eq:eq2-3}, we obtain
\begin{eqnarray}\label{eq:eq2-7}\frac{\partial \alpha_i}{\partial r_j}&=&\frac{\sinh r_k\sinh r_l}{\sqrt{Q(\br)}\sinh(r_i+r_j+r_k)\sinh(r_i+r_j+r_l)}\times\nonumber\\&&\Big[2-(y_k-y_l)^2+y_i(y_j+y_k+y_l)+y_j(y_i+y_k+y_l)\Big]
\end{eqnarray} By this formula, we prove the following lemma, which was derived in the Euclidean setting by Glickenstein, see \cite[Corollary~5]{G2}.
\begin{lemma}\label{lem:partialpositive} For a real tetrahedron $\tau(\br),$ let $r_i=\min\{r_i,r_j,r_k,r_l\}.$ Then
$$\frac{\partial \alpha_i}{\partial r_j}>0,\frac{\partial \alpha_i}{\partial r_k}>0,\frac{\partial \alpha_i}{\partial r_l}>0.$$
\end{lemma}
\begin{proof} Without loss of generality, we show that $\frac{\partial \alpha_i}{\partial r_j}>0.$ By the assumption, $y_i=\max\{y_i,y_j,y_k,y_l\},$ which yields that $$y_i(y_j+y_k+y_l)\geq (y_k-y_l)^2.$$ This implies that
\begin{eqnarray*}2-(y_k-y_l)^2+y_i(y_j+y_k+y_l)+y_j(y_i+y_k+y_l)\geq 2.
\end{eqnarray*}
The lemma follows from \eqref{eq:eq2-7}.
\end{proof} We will use this lemma to derive a comparison principle for solid angles of tetrahedra, see Theorem~\ref{thm:compare}.

The following monotonicity is useful for our applications.
\begin{proposition}\label{prop:regulartetra} The function $t\mapsto\alpha_1(t\pmb{1})$ is decreasing in $(0,\infty),$ and
$$\lim_{t\to 0}\alpha_1(t\pmb{1})=\alpha_1^{\EE}(\pmb{1}),\quad \lim_{t\to \infty}\alpha_1(t\pmb{1})=0,$$ where $\alpha_1^{\EE}(\pmb{1}):=3\arccos{1/3}-\pi$ denotes the solid angle in a regular Euclidean tetrahedron of side length $2.$
\end{proposition}
\begin{proof}
Set $f(t):=\cos \beta_{12}(t\pmb{1}).$ By \eqref{eq:eq2-5},
$$f(t)=\frac{\cosh(2t)}{1+2\cosh(2t)}.$$ By taking the derivative,
$$f'(t)=\frac{2\sinh(2t)}{(1+2\cosh(2t))^2}> 0.$$ Hence $f(t)$ is strictly increasing and $\beta_{12}(t\pmb{1})$ is strictly decreasing.
By the symmetry for $\br=t\pmb{1},$
$$h(t)=\alpha_1(t\pmb{1})=3\beta_{12}(t\pmb{1})-\pi$$ which is decreasing. Hence the limits exist as $t\to0$ and $t\to \infty,$ and the results follow from direct computation.
\end{proof}



\subsection{Virtual tetrahedra} By the change of variables in Definition~\ref{def:rtoy},
$$Q(\br)=\left(\sum_{i=1}^4 y_i\right)^2-2\sum_{i=1}^4 y_i^2+4.$$

We consider the solutions to $Q(\br)\leq 0.$ Then for any $1\leq i\leq 4,$ either
\begin{eqnarray}\label{eq:eq2-1}&&y_i\geq y_j+y_k+y_l+2(y_jy_k+y_ky_l+y_ly_j+1)^{\frac12},\quad \mathrm{or},\nonumber\\
&&y_i\leq y_j+y_k+y_l-2(y_jy_k+y_ky_l+y_ly_j+1)^{\frac12}.\end{eqnarray} For any $i,$ we set
$$f_i(y_j,y_k,y_l)=y_j+y_k+y_l+2(y_jy_k+y_ky_l+y_ly_j+1)^{\frac12}$$
and denote by
\begin{equation}
D_i=\{\br\in\R_{+}^4:y_i\geq f_i(y_j,y_k,y_l)\},
\end{equation} the $i$-th virtual tetrahedron space. By the definition, $$D_i\subset \R^4_+\setminus \Omega_{1234},\quad\forall 1\leq i\leq 4.$$
One can prove that $D_i$ is simply connected.

\begin{proposition}\label{lemma-Di-imply-ri-small}
In the $i$-th virtual space $D_i$, $r_i<\min\{r_j,r_k,r_l\}.$
\end{proposition}
\begin{proof}
This follows from
$$y_i\geq f_i(y_j,y_k,y_l)>\max\{y_j,y_k,y_l\}.$$
\end{proof} This proposition implies that $\{D_i\}_{i=1}^4$ are mutually disjoint.

\begin{proposition}\label{decomp}
$$\R_{+}^4-\Omega_{1234}=D_1\cup D_2\cup D_3\cup D_4.$$
\end{proposition}
\begin{proof} It suffices to show that $\R_{+}^4-\Omega_{1234}\subset \cup_{i=1}^4D_i.$ Suppose it is not true, then there is a $\br\in \R_{+}^4\setminus(\Omega_{1234}\cup (\cup_{i=1}^4D_i)).$ Then $Q(\br)\leq 0.$ Without loss of generality, we assume that $y_1\geq y_2\geq y_3\geq y_4.$ By \eqref{eq:eq2-1} for $i=1,$
$$y_1\leq y_2+y_3+y_4-2(y_2y_3+y_3y_4+y_4y_2+1)^{\frac12}.$$ This yields that
$$y_3+y_4\geq 2(y_2y_3+y_3y_4+y_4y_2+1)^{\frac12}.$$ By taking the square in both sides of the inequality, we get
$$(y_3-y_4)^2\geq 4 y_2(y_3+y_4)\geq 2(y_3+y_4)^2.$$ This is a contradiction and the proposition follows.
\end{proof}

The next corollary follows directly from above propositions, and hence we omit its proof here. 
\begin{corollary}\label{lemma-ri-small-imply-Di}
If $Q(\br)\leq0$, then $\{r_1,r_2,r_3, r_4\}$ have a strictly minimal value. Moreover, if $\{r_1,r_2,r_3, r_4\}$ attains its strictly minimal value at $r_i$ for some $i\in\{1,2,3,4\}$, then $r\in D_i$.
\end{corollary}

The decomposition has an interesting geometric intepretation, see \cite{Xu,GJS}.
For $\{i,j,k,l\}=\{1,2,3,4\},$ we place three balls $B_j$, $B_k$ and $B_l$ with radii $r_j$, $r_k$ and $r_l$ in $\HH^3$, externally tangent to each other, whose centers lie on an embedded totally geodesic hyperbolic plane $\Pi.$ For $\br\in D_i,$ the radius of the fourth ball $B_i$ satisfies $r_i<\min\{r_j,r_k,r_l\}.$ Geometrically, for $Q(\br)=0,$ the center of the fourth ball $B_i$ lies on the $\Pi$ and the intersections, $B_i\cap \Pi,B_j\cap \Pi,B_k\cap \Pi$ and $B_l\cap \Pi,$ are mutually externally tangent disks in $\Pi.$ For $Q(\br)<0,$ the ball $B_i$ goes through the gap between the other three mutually tangent balls, which cannot be realized in $\HH^3,$ and hence we call it a virtual tetrahedron.

\subsection{Extensions}
\subsubsection{A $C^0$-extension of solid angles}
Solid angles are initially defined for real tetrahedra. There is a natural way to extend them to all tetrahedra including virtual ones. Bobenko, Pinkall, Springborn \cite{Bobenko} introduced the extension method for angles of all triangles including degenerate ones to establish a variational principle, which connects Milnor's Lobachevsky volume function of decorated hyperbolic ideal tetrahedra to Luo's discrete conformal changes \cite{L1}. Luo \cite{L2} systematically developed their extension idea and proved some rigidity results related to inversive distance circle packings and discrete conformal factors, see e.g.  \cite{GJ1,GJ3}. The extension of dihedral angles in a $3$-dimensional decorated ideal (or hyper-ideal) hyperbolic polyhedral first appeared in Luo and Yang's work \cite{LuoYang}. They proved the rigidity of hyperbolic cone metrics on $3$-manifolds which are isometric gluing of ideal and hyper-ideal tetrahedra in hyperbolic spaces.

For real and virtual tetrahedra,  Xu \cite{Xu} introduce a natural extension of solid angles. For a tetrahedron $\tau=\{1234\},$ the extended solid angle $\wt{\alpha_i}$ at the vertex $i$ is defined as
\begin{equation}\label{eq:extsolid}\wt{\alpha_i}(\br)=\left\{\begin{array}{ll} {\alpha_i}(\br),&\br\in \Omega_{1234},\\
2\pi,&\br\in D_i,\\
0,&\br\in D_j\cup D_k\cup D_l.\\
\end{array}\right.\end{equation}This extends the definitions of solid angles to all tetrahedra parametrized by $\R^4_+,$ which are piecewise constant on virtual tetrahedra. Xu \cite[Lemma~2.6]{Xu} claimed that this extension $\wt{\alpha_i}$ is continuous on $\R^4_+$ in the hyperbolic setting by some geometric intuition. In the Euclidean setting, the continuity of $\wt{\alpha_i}$ was rigorously proved by Glickenstein \cite[Proposition~6]{G2}. In the hyperbolic setting, we prove the counterpart of Glickenstein's result, which will imply the continuity of the extended solid angles, see Theorem~\ref{thm:cont}.

\begin{proposition}\label{prop:approx} Let $\overline{\br}\in D_i$ and $Q(\overline{\br})=0.$ Then for any sequence $\{\br^{(n)}\}_{n=1}^\infty$ in $\Omega_{1234}$ converging to $\overline{\br},$
\begin{eqnarray*} &&\alpha_i(\br^{(n)})\to 2\pi,\\ &&\max\{\alpha_j(\br^{(n)}),\alpha_k(\br^{(n)}),\alpha_l(\br^{(n)})\}\to 0, \ n\to \infty.
\end{eqnarray*}
\end{proposition}
\begin{proof} By the assumption, $\br^{(n)}\to \overline{\br}$ as $n\to \infty.$ Without loss of generality, we prove that  $$\alpha_i\to 2\pi,\quad \alpha_j\to 0,\quad n\to \infty.$$ By \eqref{eq:eq2-3}, it suffices to show that
$$\beta_{ij},\beta_{ik},\beta_{il}\to \pi,\quad \beta_{jk},\beta_{jl}\to 0.$$ By the symmetry, we only need to show that $\beta_{ij}\to \pi$ and $\beta_{jk}\to 0.$

By \eqref{eq:eq2-5},
$$\sin^2(\beta_{ij}(\br))=Q(\br)/(4\sinh r_k\sinh r_l\sinh^2(r_i+r_j)\sinh(r_i+r_j+r_k)\sinh(r_i+r_j+r_l))^{\frac12}.$$
So that by $Q(\overline{\br})=0,$
\begin{equation}\label{eq:eq2-6}\sin^2(\beta_{ij}(\br^{(n)}))\to 0,\quad n\to \infty.\end{equation} Similarly, we have
\begin{equation}\label{eq:eqadd-1}\sin^2(\beta_{jk}(\br^{(n)}))\to 0,\quad n\to \infty. \end{equation}

For $\overline{\br}=(\overline{r}_i,\overline{r}_j,\overline{r}_k,\overline{r}_l)\in D_i,$ $\overline{r}_i<\min\{\overline{r}_j,\overline{r}_k,\overline{r}_l\}.$ Hence, $$Q(\overline{\br})-(\overline{y}_i+\overline{y}_j)^2+(\overline{y}_k-\overline{y}_l)^2<0.$$
Since $\br^{(n)}\to \overline{\br},$ for sufficiently large $n,$
$$Q(\br^{(n)})-(y_i^{(n)}+y_j^{(n)})^2+(y_k^{(n)}-y_l^{(n)})^2<0,$$ which implies that
$\cos(\beta_{ij}(\br^{(n)}))<0$ by \eqref{eq:eq2-5}. Hence by \eqref{eq:eq2-6}, $$\cos(\beta_{ij}(\br^{(n)}))\to -1,$$ which yields that $\beta_{ij}(\br^{(n)}))\to \pi$ as $n\to \infty.$

Moreover, for $\overline{\br}\in D_i,$
\begin{equation} \overline{y}_i\geq \overline{y}_j+\overline{y}_k+\overline{y}_l+2(\overline{y}_j\overline{y}_k+\overline{y}_k\overline{y}_l+\overline{y}_l\overline{y}_j+1)^{\frac12}.
\end{equation} This implies that
\begin{eqnarray*}&&Q(\overline{\br})-(\overline{y}_k+\overline{y}_l)^2+(\overline{y}_i-\overline{y}_j)^2\\&\geq&
-(\overline{y}_k+\overline{y}_l)^2+(\overline{y}_k+\overline{y}_l+2(\overline{y}_j\overline{y}_k+\overline{y}_k\overline{y}_l+\overline{y}_l\overline{y}_j+1)^{\frac12})^2>0.\end{eqnarray*} This yields that for sufficiently large $n,$
$$Q(\br^{(n)})-(y_k^{(n)}+y_l^{(n)})^2+(y_i^{(n)}-y_j^{(n)})^2>0,$$
which implies $\cos(\beta_{jk}(\br^{(n)}))>0.$ This proves $\beta_{jk}(\br^{(n)})\to 0$ by \eqref{eq:eqadd-1}.
Hence the proposition follows.
\end{proof}
\begin{remark} In the above proof, we directly calculate dihedral angles and solid angles by \eqref{eq:eq2-5}. These explicit formulae help us to analyze the behaviors of angles when ball packings tend to degenerate configurations. Our proof is different from Glickenstein's \cite[Proposition~6]{G2}.
\end{remark}

The above proposition can be used to show the continuity of extended solid angles.
\begin{theorem}\label{thm:cont}
For each vertex $i\in\{1,2,3,4\}$, the extended solid angle $\wt{\alpha}_i$, defined on $\R^4_{+}$, is a continuous extension of $\alpha_i$ on $\Omega_{1234}.$
\end{theorem}
\begin{proof} By the definition \eqref{eq:extsolid}, it is obvious that $\wt{\alpha}_i$ is continuous in $\Omega_{1234}$ and the interiors of $\cup_{m=1}^4D_m.$ For $1\leq m\leq 4,$ and any $\br \in \partial D_m\cap \R^4_+,$ the continuity of $\wt{\alpha}_i$ at $\br$ follows from Proposition~\ref{prop:approx}. This proves the theorem.
\end{proof}

\subsubsection{Extended Cooper-Rivin functional on tetrahedra}
Since solid angles have been extended to all tetrahedra, we may extend the Cooper-Rivin's functional $\mathcal{U}$ defined in \eqref{eq:crivin} to all tetrehedra.

We follow the approach pioneered by Luo \cite{L2}. A differential $1$-form $\omega=\sum_{i=1}^n a_i(\pmb{x})dx_i$ in an open set $V\subset\mathds{R}^n$ is said to be continuous if each $a_i(\pmb{x})$ is a continuous function on $V$. A continuous $1$-form $\omega$ is called closed if $\int_{\partial \tau} \omega=0$ for any Euclidean triangle $\tau\subset V$. By the standard approximation theory, if $\omega$ is closed and $\gamma$ is a piecewise $C^1$-smooth null-homotopic loop in $V$, then $\int_\gamma \omega=0$. If $V$ is simply connected, then the integral $F({\pmb x})=\int_{\pmb{x_0}}^{\pmb{x}}\omega$ is well defined (where $\pmb{x_0}\in V$ is arbitrarily chosen), independent of the choice of piecewise smooth paths in $V$ from $\pmb{x_0}$ to $\pmb{x}$. Moreover, the function $F(\pmb{x})$ is $C^1$-smooth and $\frac{\partial F(\pmb{x})}{\partial x_i}=a_i(\pmb{x})$. Luo established the following fundamental $C^1$-smooth and convex extension theory.
\begin{lemma}
\label{lemma-luo's-extension}
(Luo's convex $C^1$-extension, \cite{L2})
Suppose $X\subset \mathds{R}^n$ is an open convex set and $A\subset X$ is an open and simply connected subset of $X$ bounded by a real analytic codimension-1 submanifold in $X$. If $\omega=\sum_{i=1}^n a_i(\pmb{x})dx_i$ is a continuous closed $1$-form on $A$ so that $F(\pmb{x}):=\int_{\pmb{x_0}}^{\pmb{x}}\omega$ is locally convex (concave resp.) on $A,$ and each $a_i$ on $A$ can be extended to a continuous function $\wt{a}_i$ on $X,$ which is piecewise constant on connected components $X\setminus A,$ then $\widetilde F(\pmb{x}):=\int_{\pmb{x_0}}^{\pmb{x}} \wt{a}_i(\pmb{x})dx_i$ is a $C^1$-smooth convex (concave resp.) function on $X$ extending $F$.
\end{lemma}

In our setting, for a tetrahedron $\tau=\{1234\}$ we define a continuous 1-form on $\Omega_{1234},$
$$\omega(\br)=\sum_{i=1}^4\alpha_i(\br) dr_i,$$ which is closed by \eqref{eq:1-1}. It extends to a continuous 1-form on $\R^4_+$ as
\begin{equation}\label{eq:wtomega}\wt{\omega}(\br)=\sum_{i=1}^4\wt{\alpha}_i(\br) dr_i,\end{equation} which is piecewise constant on connected components of $\R^4_+\setminus \Omega_{1234}.$ By Lemma~\ref{lemma-luo's-extension}, the following functional
\begin{equation}\label{eq:3-1}\wt{\mathcal{U}}(r)=\mathcal{U}(\pmb{1})+\int_{\pmb{1}}^{\br} \wt{\omega},\end{equation}
is a $C^1$-smooth concave functional on $\R^4_+,$ where $\mathcal{U}$ is the Cooper-Rivin's functional. By Lemma~\ref{lemma-luo's-extension}, for any $\br\in \Omega_{1234},$
$$\wt{\mathcal{U}}(r)=\mathcal{U}(\pmb{1})+\int_{\pmb{1}}^{\br}{\omega}.$$ By taking the integral along a path in $\Omega_{1234},$ noting that \eqref{eq:1-1},
$$\wt{\mathcal{U}}(r)=\mathcal{U}(r),\quad \forall\ \br\in \Omega_{1234}.$$ Combining these facts with \eqref{eq:secondder}, we have the following lemma.
\begin{lemma}\label{lem:extension} For a tetrahedron $\tau=\{1234\}, $ $$\wt{\mathcal{U}}:\R^4_+\to\R,$$ defined in \eqref{eq:3-1}, is a $C^1$-smooth concave functional, extending the Cooper-Rivin's functional $\mathcal{U}$ on $\Omega_{1234},$ which is $C^\infty$-smooth and strictly concave on $\Omega_{1234}.$ Moreover,
\begin{eqnarray*}&&\nabla \wt{\mathcal{U}}(\br)=\wt{\pmb\alpha}(\br),\quad  \forall\ \br\in \R^4_+,\\
&&\nabla^2 \wt{\mathcal{U}}(\br)=\frac{\partial \pmb\alpha}{\partial \br}(\br),\quad  \forall\ \br\in \Omega_{1234}.\end{eqnarray*}
\end{lemma}

\subsection{Comparison principles}
The following lemma in the Euclidean setting was proved by Glickenstein \cite[Lemma~7]{G2}.
We prove it in hyperbolic setting essentially following his proof strategy.
\begin{lemma}\label{lem:Glick1} For a tetrahedron $\tau=\{1234\}$ and any ball packing $\br\in \R^4_+,$
$$r_i\leq r_j\quad \mathrm{if\ and\ only\ if}\quad \wt{\alpha}_i\geq \wt{\alpha}_j.$$
\end{lemma}
\begin{proof} The lemma is equivalent to the following two statements:
\begin{enumerate}[(a)]
\item $r_1= r_2\quad \mathrm{implies}\quad \wt{\alpha}_1= \wt{\alpha}_2,$ and
\item $r_1< r_2\quad \mathrm{implies}\quad \wt{\alpha}_1> \wt{\alpha}_2.$
\end{enumerate} The statement $(a)$ is immediate. In fact, it is trivial if $\br=(r_1,r_2,r_3,r_4)$ is virtual. For the case of $\br\in \Omega_{1234},$ it follows from an isometry between tetrahedra $\tau(r_1,r_2,r_3,r_4)$ and $\tau(r_2,r_1,r_3,r_4).$

For the statement $(b),$ it suffices to consider $\br\in \Omega_{1234},$ otherwise it is trivial by Proposition~\ref{lemma-Di-imply-ri-small} and Corollary~\ref{lemma-ri-small-imply-Di}. We define a path $$\pmb\sigma(s):=((1-s)r_1+sr_2,(1-s)r_2+sr_1,r_3,r_4),\quad s\in [0,1],$$ which connects the points $(r_1,r_2,r_3,r_4)$ and $(r_2,r_1,r_3,r_4).$ Consider the functional $\wt{\mathcal{U}}$ on $\R^4_+.$ Since $\mathcal{U}(r_1,r_2,r_3,r_4)$ is symmetric w.r.t. the permutations, i.e. for any permutation $\rho$ of $\{1,2,3,4\},$
$$\mathcal{U}(r_{\rho(1)},r_{\rho(2)},r_{\rho(3)},r_{\rho(4)})=\mathcal{U}(r_{1},r_{2},r_{3},r_{4}).$$ By the definition \eqref{eq:3-1} and the symmetry of $\wt{\omega}$ in \eqref{eq:wtomega}, $\wt{\mathcal{U}}$ enjoys the same symmetry. Hence the function
$f(s):=\wt{\mathcal{U}}(\pmb\sigma(s))$ on $[0,1]$ is symmetric w.r.t. $\frac12,$ i.e.
\begin{equation}\label{eq:Gl-1}f(s)=f(1-s),\quad \forall s\in [0,1].\end{equation} Since $\pmb\sigma(s)$ is linear in $s,$ by the concavity of $\wt{\mathcal{U}}$ in Lemma~\ref{lem:extension},
\begin{equation}\label{eq:Gl-2}f(s) \mathrm{\ is\ concave\ in\ } [0,1].\end{equation} Moreover, since $\pmb\sigma(0)=\br\in \Omega_{1234},$ there exists $\epsilon>0$ such that
\begin{equation}\label{eq:Gl-3}f(s) \mathrm{\ is\ strictly\ concave\ in\ } (0,\epsilon).\end{equation} Hence by \eqref{eq:Gl-1} and \eqref{eq:Gl-2},
$$f\left(\frac12\right)=\max_{s\in [0,1]}f(s).$$

We claim that $f'(0)>0.$ Suppose it is not true, i.e. $f'(0)\leq 0.$ Note that the left derivative of $f,$ $f'_-(s),$ is non-increasing and $f'_-(\frac12)\geq 0.$ This implies that $f'_-(s)=0,$ for all $s\in [0,\frac12].$ This implies that $f$ is constant on $[0,\frac12].$ This contradicts to \eqref{eq:Gl-3}. This proves the claim.
By the calculation,
$$f'(s)=(\wt{\alpha}_1(\pmb{\sigma}(s))-\wt{\alpha}_2(\pmb{\sigma}(s)))(r_2-r_1).$$ By setting $s=0,$ the lemma follows from the claim.

\end{proof}

Now we are ready to prove some comparison principles for solid angles.
\begin{theorem}\label{thm:compare} For a real or virtual tetrahedron $\tau(\br),$ if $r_i=\min\{r_i,r_j,r_k,r_l\},$ then
$$\wt{\alpha}_i(\br)\geq \alpha_i(r_i\pmb{1}).$$
\end{theorem}
\begin{proof} It suffices to consider the case that $\br\in \Omega_{1234}.$ Otherwise, by Proposition~\ref{lemma-Di-imply-ri-small} and Corollary~\ref{lemma-ri-small-imply-Di} the result is trivial since
$\wt{\alpha}_i(\br)=2\pi.$



Let $\pmb\sigma:[0,1]\to \R^4_+$ be a curve defined as
$$\pmb\sigma(s)=(\sigma_i(s),\sigma_j(s),\sigma_k(s),\sigma_l(s)):=(1-s)\br+sr_i\pmb{1},\quad \forall s\in [0,1],$$ which connects the points $\br$ and $r_i\pmb{1}.$

We claim that $\pmb\sigma(s)\in \Omega_{1234}$ for all $s\in [0,1].$ Suppose that it is not true, then
$$A:=\{s\in(0,1): \pmb \sigma(s)\in \R^4_+\setminus\Omega_{1234}\}\neq \emptyset.$$ Let $s_0=\inf A.$ Since $\Omega_{1234}$ is open, $A$ is closed in $(0,1),$ which implies that $s_0\in A$ and $s_0>0.$ Consider the function $\alpha_i(\pmb\sigma(s))$ on $(0,s_0).$ By the definition of $s_0,$ $\pmb\sigma(s)\in \Omega_{1234}$ for any $s\in(0,s_0).$ Note that for any $s\in(0,s_0),$ $\sigma_i(s)=r_i=\min\{\sigma_i(s),\sigma_j(s),\sigma_k(s),\sigma_l(s)\}.$ Hence by Lemma~\ref{lem:partialpositive}
\begin{eqnarray}\label{eq:eq4-5}&&\frac{d}{ds}(\alpha_i(\pmb\sigma(s)))\\&=&(r_i-r_j)\frac{\partial \alpha_i}{\partial r_j}(\pmb\sigma(s))+(r_i-r_k)\frac{\partial \alpha_i}{\partial r_k}(\pmb\sigma(s))+(r_i-r_l)\frac{\partial \alpha_i}{\partial r_l}(\pmb\sigma(s))\nonumber\\
&<&0.\nonumber\end{eqnarray} Hence $\alpha_i(\pmb\sigma(s))$ is decreasing in $(0,s_0).$ This yields that
$$2\pi>\alpha_i(\pmb\sigma(0))\geq \lim_{s\to s_0}\alpha_i(\pmb\sigma(s))=\wt{\alpha}_i(\pmb\sigma(s_0))=2\pi.$$ This is a contradiction and proves the claim.

By the claim and the same argument as in the proof of \eqref{eq:eq4-5}, one can show that for any $s\in (0,1),$
$$\frac{d}{ds}(\alpha_i(\pmb\sigma(s)))<0.$$ This yields that
$$\alpha_i(\pmb\sigma(0))\geq{\alpha}_i(\pmb\sigma(1)),$$ which proves the theorem.


\end{proof}

\begin{theorem}\label{thm:secondcompare} For a real or virtual tetrahedron $\tau(\br),$ if $r_i=\max\{r_i,r_j,r_k,r_l\},$ then
\begin{equation}\label{eq:compare2-1}\wt{\alpha}_i(\br)\leq \alpha_1^\EE(\pmb{1}).\end{equation}
\end{theorem}
\begin{proof} For any fixed $\overline{\br}\in \R^4_+$ with $\overline{r}_i=\max\{\overline{r}_i,\overline{r}_j,\overline{r}_k,\overline{r}_l\},$ we will prove that
$$\wt{\alpha}_i(\overline{\br})\leq \alpha_1^\EE(\pmb{1}).$$
It suffices to consider the case that $\overline{\br}\in \Omega_{1234}.$ Otherwise, for $\overline{\br}\in \R^4_+\setminus \Omega_{1234},$ by Proposition~\ref{lemma-Di-imply-ri-small} and Corollary~\ref{lemma-ri-small-imply-Di} the above inequality trivially holds since
$\wt{\alpha}_i(\overline{\br})=0.$

Let $\overline{\br}=(\overline{r}_1,\overline{r}_2,\overline{r}_3,\overline{r}_4)\in \Omega_{1234}.$ Let $\wt{\mathcal{U}}$ be the functional defined in \eqref{eq:3-1}. For any $\delta>0,$ we define a functional on $\R^4_+$ as follows,
$$\varphi_\delta(\br):=\wt{\mathcal{U}}(\br)-\left(\sum_{m=1}^4\alpha_m(\delta \pmb{1}) r_m+2\vol(\delta\pmb{1})\right).$$
 By Lemma~\ref{lem:extension}, $\varphi_\delta$ is a $C^1$-smooth concave function on $\R^4_+,$ and
$$\nabla \varphi_\delta(\br)=\wt{\pmb\alpha}(\br)-\pmb\alpha(\delta\pmb{1}).$$ At the point $\br=\delta\pmb{1},$
$$\varphi_\delta(\delta\pmb{1})=0\quad \mathrm{and}\quad \nabla\varphi_\delta(\delta\pmb{1})=\pmb{0}.$$ Hence the function $\varphi_\delta$ attains its maximum at $\delta\pmb{1},$ and
$$\varphi_\delta(\br)\leq 0,\quad\forall\ \br\in \R^4_+.$$ By taking $\br=\overline{\br},$ we have
\begin{eqnarray*}\left(\sum_{m=1}^4\alpha_m(\overline{\br})\overline{r}_m+2\vol(\overline{\br})\right)-\left(\sum_{m=1}^4\alpha_m(\delta \pmb{1}) \overline{r}_m+2\vol(\delta\pmb{1})\right)\leq 0.
\end{eqnarray*}
For the fixed $\overline{\br},$ there is a constant $\delta_0>0$ such that for any $\delta\leq \delta_0,$
$$\vol(\delta\pmb{1})\leq \vol(\overline{\br}).$$ Hence by the above results,
$$\sum_{m=1}^4\alpha_m(\overline{\br}) \overline{r}_m\leq \sum_{m=1}^4\alpha_m(\delta \pmb{1}) \overline{r}_m=\alpha_1(\delta\pmb{1})\sum_{m=1}^4\overline{r}_m.$$ For $\overline{r}_i=\max\{\overline{r}_i,\overline{r}_j,\overline{r}_k,\overline{r}_l\},$ by Lemma~\ref{lem:Glick1}, $${\alpha}_i(\overline{\br})=\min\{{\alpha}_i(\overline{\br}),{\alpha}_j(\overline{\br}),{\alpha}_k(\overline{\br}),{\alpha}_l(\overline{\br})\}.$$ This yields that
$${\alpha}_i(\overline{\br})\leq \frac{\sum_m {\alpha_m}(\overline{\br}) \overline{r}_m}{\sum_m \overline{r}_m}\leq \alpha_1(\delta\pmb{1}).$$ By passing to the limit, $\delta\to 0,$ the theorem follows from Proposition~\ref{prop:regulartetra}.

\end{proof}

\section{Ball packings on triangulations and combinatorial Yamabe flows}\label{s:ballpackings}
\subsection{Glickenstein's flow in hyperbolic background geometry}
\label{subsec:g-flow}
For a triangulated 3-manifold $(M,\mathcal{T}),$ we consider the hyperbolic combinatorial Yamabe flow \eqref{hyperyflow}, which is an analog to Glickenstein's combinatorial Yamabe flow \eqref{Def-Flow-Glickenstein} in the Euclidean setting.
Note that the set of real ball packings $\MT$ is open and simply-connected subset in $\R^N_+$. 

In $\mathcal{M}_{\mathcal{T}}$, the terms on the right hand side of \eqref{hyperyflow}, $-K_i\sinh r_i,$ as a function of $\br=(r_1,\cdots,r_N)$ are smooth and hence
locally Lipschitz continuous. By Picard theorem in classical ODE theory, the flow \eqref{hyperyflow} has
a unique solution $r(t)$, $t\in[0,\epsilon)$ for some $\epsilon>0$. As a consequence, we yield the following proposition.
\begin{proposition}\label{prop:combflow}
Given a triangulated 3-manifold $(M, \mathcal{T}),$ for any initial ball packing $\br(0)\in\mathcal{M}_{\mathcal{T}}$, the solution $\{\br(t)\}\subset\mathcal{M}_{\mathcal{T}}$ to the flow \eqref{hyperyflow}
exists and is unique on the maximal existence interval $[0, T)$ with $0<T\leq+\infty$.
\end{proposition}

For a triangulated 3-manifold $(M,\mathcal{T}),$ we define a functional on the set of all real ball packings  as
$${\mathcal{S}}(\br):=\sum_{i\in  \mathcal{T}_0}4\pi r_i-\sum_{\{ijkl\}\in \mathcal{T}_3}{\mathcal{U}}(\br),$$ where ${\mathcal{U}}(\br)$ is given in \eqref{eq:crivin}. Hence,
for any $\br\in \MT,$
\begin{eqnarray*}{\mathcal{S}}(\br)&=&\sum_{i\in \mathcal{T}_0}\left(4\pi-\sum_{\{ijkl\}\in \mathcal{T}_3}{\alpha}_{ijkl}(\br)\right)-2\vol_M(\br)\\&=&\sum_{i\in \mathcal{T}_0}{K}_i r_i-2\vol_M(\br),
\end{eqnarray*} where $\vol_M(\br)=\sum_{\{ijkl\}\in \mathcal{T}_3}\vol_{ijkl}(\br)$ denotes the summation of the volumes of tetrahedra in the triangulation in ball packing metric $\br.$ This means that ${\mathcal{S}}$ coincides with the Cooper-Rivin's functional introduced in \eqref{eq:cprivinf}, so that we use the same notation for them.
Moreover, one can show that
\begin{equation}\label{eq:nabequal}\nabla{\mathcal{S}}(\br)={\pmb{K}}(\br),\quad \br\in \MT,\end{equation} and ${\mathcal{S}}$ is a $C^\infty$-smooth strictly convex functional on $\MT.$
\begin{proposition}\label{prop:decreasing-1} The functional ${\mathcal{S}}$ is non-increasing under the flow \eqref{hyperyflow}, i.e. for any solution $\br(t)$ to the flow \eqref{hyperyflow},
$$\frac{d}{dt}\mathcal{S}(\br(t))\leq 0.$$
\end{proposition}
\begin{proof} By direct calculation,
$$\frac{d}{dt}\mathcal{S}(\br(t))=-\sum_i K_i^2\sinh(r_i(t))\leq 0.$$
\end{proof}

Suppose that the solution to the flow \eqref{hyperyflow} converges, see \eqref{eq:converges1} for the definition, then the limit ball packing has vanishing combinatorial scaler curvature.
\begin{proposition}\label{prop-converg-imply-const-exist}
Let $\br(t)$ be a solution to the flow \eqref{hyperyflow} which converges to $\overline{\br}\in \MT.$ Then
$$\pmb{K}(\overline{\br})=0.$$
\end{proposition}
\begin{proof}
This is well-known in classical ODE theory. For the convenience of readers, we include the proof here.
For any $t>0,$ by Proposition~\ref{prop:decreasing-1}, $\mathcal{S}(\br(t))$ is non-increasing. Moreover,
$$\{\mathcal{S}(\br(t)):t\geq 0\}$$ is bounded from below, since $\mathcal{S}$ is continuous on $\MT$ and $\br(t)\to \overline{\br},$ $t\to \infty.$
Hence the following limit exists and is finite, $$\lim_{t\to\infty}{\mathcal{S}}(\br(t))=C.$$ Consider the sequence $\{{\mathcal{S}}(\br(n))\}_{n=1}^\infty.$ By the mean value theorem, for any $n\geq 1$ there exists $t_n\in(n,n+1)$ such that
\begin{eqnarray}\label{eq:eq1-5-1}{\mathcal{S}}(\br(n+1))-{\mathcal{S}}(\br(n))&=&\frac{d}{dt}\Big|_{t=t_n}({\mathcal{S}}(\br(t)))\nonumber\\
&=&-\sum_i {K}^2_i(\br(t_n))\sinh(r_i(t_n)).
\end{eqnarray} Note that $$\lim_{n\to\infty}{\mathcal{S}}(\br(n+1))-{\mathcal{S}}(\br(n))=C-C=0.$$
Hence by \eqref{eq:eq1-5-1},
$$\lim_{n\to \infty} {K}_i(\br(t_n))=0,\quad\forall \ i\in \mathcal{T}_0.$$
Since
$\br(t_n)\to \overline{\br}$ as $n\to\infty,$
the continuity of ${K}_i$ yields that
$${K}_i(\overline{\br})=0,\quad \forall  \ i\in \mathcal{T}_0.$$ This proves the proposition.
\end{proof}


\subsection{The extended Cooper-Rivin's functional}
Let $(M,\mathcal{T})$ be a triangulated 3-manifold. Using extended solid angles as in \eqref{eq:extsolid}, we define the extended combinatorial scaler curvature for all ball packings: for any $\br\in \R^N_+$ and any $i\in \mathcal{T}_0,$
\begin{equation}\label{eq:extendK}\wt{K}_i(\br):=4\pi-\sum_{\{ijkl\}\in\mathcal{T}_3}\wt{\alpha}_{ijkl}.\end{equation}

We define the extended Cooper-Rivin's functional on the set of all ball packings on $(M,\mathcal{T}).$ For any ball packing $\br\in \R^N_+,$ the extended Cooper-Rivin's functional is given by
$$\wt{\mathcal{S}}(\br)=\sum_{i\in  \mathcal{T}_0}4\pi r_i-\sum_{\{ijkl\}\in \mathcal{T}_3}\wt{\mathcal{U}}(\br),$$ where $\wt{\mathcal{U}}(\br)$ is defined in \eqref{eq:3-1}. Hence by Lemma~\ref{lem:extension} and \eqref{eq:crivin}, we have
$$\wt{\mathcal{S}}(\br)={\mathcal{S}}(\br),\quad \forall \ \br\in \MT.$$
Moreover, one can prove that
\begin{equation}\label{eq:criticalpoint}\nabla \wt{\mathcal{S}}(\br)=\wt{\pmb{K}}(\br).\end{equation} Hence, the critical points of the functional $\wt{\mathcal{S}}$ are given by ball packings with vanishing extended combinatorial scaler curvature. In addition, $\wt{\mathcal{S}}$ is a $C^1$-smooth convex functional on $\R^N_+$
 which is $C^\infty$-smooth strictly convex on $\MT.$ We summarize them in the following theorem.
 \begin{theorem}\label{thm:extendedf}For a triangulated 3-manifold $(M,\mathcal{T}),$ the extended Cooper-Rivin's functional $\wt{\mathcal{S}}:\R^N_+\to\R$ is a $C^1$-smooth convex functional, extending $\mathcal{S}$ in \eqref{eq:cprivinf}, which is $C^\infty$-smooth strictly convex on $\MT.$ Moreover,
 $$\nabla \wt{\mathcal{S}}(\br)=\wt{\pmb{K}}(\br),\quad \forall\ \br\in \R^N_+.$$
  $$\nabla^2 \wt{\mathcal{S}}(\br)=\frac{\partial \pmb{K}}{\partial \br},\quad \forall\ \br\in \MT.$$
 \end{theorem}

 By using the extended Cooper-Rivin's functional, with some modification, Xu \cite{Xu} proved the following rigidity result for real ball packings.
 \begin{theorem}[Theorem~1.2 in \cite{Xu}]\label{thm:xu} For a triangulated 3-manifold $(M,\mathcal{T}),$
 the map
 $$\pmb{K}:\MT\mapsto \R^N$$ is injective.
 \end{theorem}

We prove a generalization of Xu's rigidity theorem which will be useful for our purposes.
 \begin{theorem}\label{thm:alternative} Let $(M,\mathcal{T})$ be a triangulated 3-manifold and $\br_1\in\MT.$ Suppose that there is $\br_2\in \R^N_+$ such that
 $$\wt{\pmb K}(\br_2)=\pmb K(\br_1),$$ then $\br_1=\br_2.$
 \end{theorem}
 \begin{proof} Consider the set $A:=\{\br\in\R^N_+: \wt{\pmb K}(\br)=\pmb K(\br_1)\}.$ We define the functional on $\R^N_+$ by
 $$F(\br)=\wt{\mathcal{S}}(\br)-\sum_{i\in\mathcal{T}_0}K_i(\br_1)r_i.$$ By Theorem~\ref{thm:extendedf}, the set of critical points of the functional $F$ is given by $A.$ Note that the functional $F$ is $C^1$-smooth, convex on $\R^N_+$ and strictly convex on $\MT.$ Since $\br_1\in A\cap \MT,$ a well-known result on convex functions implies that $\br_1$ is the unique critical point of $F,$ which proves the theorem.
 \end{proof}

Now we are ready to prove Theorem~\ref{thm:Yamabeequivalent}.
\begin{proof}[Proof of Theorem~\ref{thm:Yamabeequivalent}] The equivalence $(1)\Leftrightarrow (2)$ follows from \eqref{eq:nabequal}. The implication $(3)\Rightarrow (2)$ is obvious. Now we prove that $(1)\Rightarrow (3).$ By Theorem~\ref{thm:extendedf}, $\overline{\br}$ is a critical point of the functional $\wt{\mathcal{S}}$ which is $C^1$-smooth convex on $\R^N_+.$ Hence $\overline{\br}$ is a global minimizer of $\wt{\mathcal{S}}.$ Note that $\wt{\mathcal{S}}$ extends $\mathcal{S}$ on $\MT$ and $\overline{\br}\in \MT.$ This implies that $\overline{\br}$ is a global minimizer of $\mathcal{S}$ on $\MT.$ This proves the result.
\end{proof}

\subsection{Longtime existence of the extended combinatorial Yamabe flow}
\label{section-long-exit-extend-flow}
As mentioned in the introduction, the solution to the combinatorial Yamabe flow \eqref{hyperyflow} can develop singularity in finite time, i.e. the maximal time $T<\infty.$ To resolve the problem, we define the extended combinatorial Yamabe flow in \eqref{extflow} using the extension of the combinatorial scaler curvature \eqref{extended curvature}. We will prove that the solution to the flow \eqref{extflow} exists for all time and extends the original flow \eqref{hyperyflow}, see analogous results in Euclidean setting in \cite{GJ1,GJ3}. In the hyperbolic setting, we use the hyperbolic geometry to get some a priori estimates for the solutions and then obtain the long time existence of the solutions, which is quite different from the Euclidean setting.
\begin{theorem}[Long time existence]
\label{thm-yang-write}
For any initial data $\br(0)\in \R^N_+,$ there exists a solution $\br(t)$ to the extended combinatorial Yamabe flow \eqref{extflow} which exists for all $t\in[0,\infty).$
\end{theorem}
\begin{remark} We will prove that such a solution is unique and extends the solution to the combinatorial Yamabe flow \eqref{hyperyflow} on $[0,T)$ in the next subsection.
\end{remark}

In order to estimate the solutions to extended combinatorial Yamabe flow \eqref{extflow}, we need the following calculus lemma.
For a continuous function $f:[0,\infty)\to \R$ and any $C\in\R,$ the upper level set of $f$ at $C$ is defined as
$$\{f>C\}:=\{t\in[0,\infty): f(t)>C\}.$$ The lower level set $\{f<C\}$ is defined similarly.
\begin{lemma}\label{lem:calc}
Let $f:[0,\infty)\to \R$ be a locally Lipschitz function. Suppose that there is a constant $C$ such that
$$f'(t)\leq 0,\quad \mathrm{for\ a.e.}\ t\ \mathrm{in}\ \{f>C\},$$ then
$$f(t)\leq \max\{f(0), C\},\quad \forall\ t\in [0,\infty).$$ Similarly, if $$f'(t)\geq 0, \quad \mathrm{for\ a.e.}\ t\ \mathrm{in}\ \{f<C\},$$ then
$$f(t)\geq \min\{f(0), C\},\quad \forall\ t\in [0,\infty).$$
\end{lemma}
\begin{proof} Without loss of generality, we prove the first assertion.
Suppose that it is not true, then
$\{f>D\}$ for $D:=\max\{f(0), C\}$ is a non-empty open set in $(0,\infty).$ Hence $\{f>D\}$ is a countable union of disjoint open intervals $(a_i,b_i)$ in $(0,\infty),$ i.e. $$\{f>D\}=\cup_{i=1}^\infty(a_i,b_i).$$ Consider one of these intervals, say $(a_1,b_1).$ By the continuity of $f,$ $f(a_1)=f(b_1)=D.$ Since $f$ is locally Lipschitz, for any $t\in (a_1,b_1),$ by the assumption,
$$f(t)=f(a_1)+\int_{a_1}^t f'(s)ds\leq f(a_1)=D.$$ This contradicts to $t\in \{f>D\}.$ This proves the lemma.
\end{proof}

For our purposes, we need to estimate the solid angles in the hyperbolic geometry. The following proposition is well-known in the hyperbolic geometry, see \cite[Lemma 3.5]{CL1}, \cite[Lemma 3.2]{GX4} or \cite[Lemma 2.3]{GJ3}. For a hyperbolic triangle $\Delta_{v_iv_jv_k}$ of vertices $v_i,v_j,v_k$ in $\HH^2,$ we denote by
$\gamma_{ijk}$ the angle at the vertex $v_i.$
\begin{proposition}\label{prop:stangle} For any $\epsilon>0,$ there exists a constant $C_1(\epsilon),$ depending only on $\epsilon,$ such that for any $r_i\geq C_1,r_j>0,r_k>0$ and the hyperbolic triangle $\Delta_{v_iv_jv_k}$ in $\HH^2$ with edge lengths
$$l_{v_iv_j}=r_i+r_j,l_{v_jv_k}=r_j+r_k,l_{v_kv_i}=r_k+r_i,$$
$$\gamma_{ijk}\leq \epsilon.$$
\end{proposition}

We prove the following lemma, which is crucial for the upper bound estimate of the solutions to the extended combinatorial Yamabe flow.
\begin{lemma}\label{lem:smallsolid} For any $\epsilon>0,$ there exists a constant $C_2(\epsilon)$ such that for any real tetrahedron $\tau_{v_iv_jv_kv_l}$ defined by the ball packing $\br,$ if $r_i\geq C_2,$ then
$$\alpha_{ijkl}(\br)\leq \epsilon.$$
\end{lemma}
\begin{proof} Set $r_0:=\arcsinh1.$ Let $C_2$ be the constant satisfying $C_2\geq r_0,$ to be determined later. 
For the tetrahedron $\tau_{v_iv_jv_kv_l}$ in $\HH^3,$ let $B_{r_0}(v_i)$ be the ball of radius $r_0$ centered at $v_i.$ We denote by $\Delta_{w_jw_kw_l}$ the intersection of $\partial B_{r_0}(v_i)$ and $\tau_{v_iv_jv_kv_l},$ which is a spherical triangle in the unit sphere with vertices $w_j,w_k,w_l$ on the geodesics (or edges) $v_iv_j,$ $v_iv_k$ and $v_iv_l$ respectively.
One is ready to see that
$\alpha_{ijkl}$ is equal to the area of $\Delta_{w_jw_kw_l},$ denoted by $|\Delta_{w_jw_kw_l}|.$ We write $l_{w_jw_k},l_{w_kw_l},l_{w_lw_j}$ for the lengths of sides of $\Delta_{w_jw_kw_l}$ and $$s=\frac{1}{2}(l_{w_jw_k}+l_{w_kw_l}+l_{w_lw_j}).$$ Note that$$l_{w_jw_k}=\gamma_{ijk},l_{w_kw_l}=\gamma_{ikl},l_{w_lw_j}=\gamma_{ijl},$$ where $\gamma_{ijk}$ ($\gamma_{ikl},$ $\gamma_{ijl}$ resp.) is the angle at the vertex $v_i$ of the hyperbolic triangle $\Delta_{v_iv_jv_k}$ ($\Delta_{v_iv_kv_l},$ $\Delta_{v_iv_jv_l}$ resp.). By Proposition~\ref{prop:stangle}, for any $\epsilon_1>0,$ there exists $C_1(\epsilon_1)$ such that for $r_i\geq C_1,$
$$\max\{\gamma_{ijk}, \gamma_{ikl}, \gamma_{ijl}\}<\epsilon_1.$$ Then by L'Huilier's formula in spherical geometry,
$$\tan^2\left(\frac{|\Delta_{w_jw_kw_l}|}{4}\right)=\tan\frac{s}{2}\tan\frac{s-l_{w_jw_k}}{2}\tan\frac{s-l_{w_kw_l}}{2}\tan\frac{s-l_{w_lw_j}}{2}\leq f(\epsilon_1),$$ for some function $f(\epsilon_1)$ satisfying $f(\epsilon_1)\to 0,$ as $\epsilon_1\to 0.$  Hence for any $\epsilon>0,$ we choose small $\epsilon_1$ such that for any $r_i\geq C_2:=C_1(\epsilon_1),$ $$|\Delta_{w_jw_kw_l}|<\epsilon.$$ This proves the lemma.

\end{proof}

Now we are ready to obtain upper bound estimates for the solutions to extended combinatorial Yamabe flow.
\begin{theorem}\label{thm:upperest} For a triangulated 3-manifold $(M,\mathcal{T}),$ let $\br(t)$ be a solution to the extended flow \eqref{extflow} on $[0,T),$ possibly $T=\infty.$ Then there exists a constant $C_3,$ depending on the initial data $\br(0)$ and the triangulation $\mathcal{T},$ such that
$$r_i(t)\leq C_3,\quad \forall\ i\in \mathcal{T}_0,\ t\in [0,T).$$\end{theorem}
\begin{proof} Set $f(t):=\max_{m\in \mathcal{T}_0} r_m(t).$ Then $f(t)$ is a locally Lipschitz function and for a.e. $t\in (0,\infty)$, there exists $i\in \mathcal{T}_0$ depending on $t,$ such that
\begin{equation}\label{eq:eq4-1}f(t)=r_i(t),\quad \mathrm{and}\quad f'(t)=r_i'(t).\end{equation} Let $C_2$ be the constant determined in Lemma~\ref{lem:smallsolid} such that for any real tetrahedron $\tau_{v_iv_jv_kv_l},$ if $r_i\geq C_2,$ then
\begin{equation}\label{eq:eq4-2}\alpha_{ijkl}\leq \frac{2\pi}{\max_{m\in\mathcal{T}_0}d_m}.\end{equation}

We would like to show that \begin{equation}\label{eq:eq4-3}f'(t)\leq 0,\quad \mathrm{for\ a.e.}\ t\in \{f>C_2\}.\end{equation} Let $t\in (0,\infty)$ and $i\in \mathcal{T}_0$ satisfying \eqref{eq:eq4-1}. Suppose that $t\in \{f>C_2\}.$ We claim that for any tetrahedron $\{ijkl\}$ incident to $i$ with the ball packing $\br(t),$
$$\wt{\alpha}_{ijkl}(\br(t))\leq \frac{2\pi}{\max_{m\in\mathcal{T}_0}d_m}.$$ If the tetrahedron $\{ijkl\}$ with the ball packing $\br(t)$ is real, then it follows from \eqref{eq:eq4-2}. If the tetrahedron $\{ijkl\}$ with the ball packing $\br(t)$ is virtual, then by \eqref{eq:eq4-1}, $$r_i(t)=\max\{r_i(t),r_j(t),r_k(t),r_l(t)\}.$$ This yields that $\wt{\alpha}_{ijkl}(\br(t))=0$ by the definition of $\wt{\alpha}.$ This proves the claim.
Hence by the claim
$$\wt{K}_i(\br(t))=4\pi-\sum_{\{ijkl\}\in\mathcal{T}_3}\wt{\alpha}_{ijkl}(\br(t))>2\pi.$$
Hence by \eqref{extflow},
$$f'(t)=r_i'(t)=-\wt{K}_i \sinh r_i<0.$$ This yields \eqref{eq:eq4-3}. Then the theorem follows from Lemma~\ref{lem:calc}.

 \end{proof}

By the above a priori estimate of the solutions, we can prove Theorem~\ref{thm-yang-write}.
\begin{proof}[Proof of Theorem~\ref{thm-yang-write}]
Since the terms on the right hand side of \eqref{extflow}, $-\wt{K}_i\sinh r_i,$ are continuous functions on $\mathds{R}^N_+$, by Peano's existence theorem in classical ODE theory, the extended flow \eqref{extflow} has at least one solution $\br(t)$ on some interval $[0,\varepsilon),$ for small $\epsilon>0.$ We denote by $[0,T)$ the maximal existence interval of the solution $\br(t)$ with the initial data $\br(0).$ By the equation \eqref{extflow},
$$\frac{d}{dt}\left(\ln(\tanh(\frac{r_i(t)}{2}))\right)=-\wt{K}_i.$$
Note that by the definition of $\widetilde{K}_i$, for any vertex $i,$ $$|\widetilde{K}_i|\leq 2\pi\left(\max_{i\in\mathcal{T}_0}d_i+1\right)=:C.$$  Hence $$\tanh\left(\frac{r_i(0)}{2}\right)e^{-Ct}\leq \tanh\left(\frac{r_i(t)}{2}\right)\leq \tanh\left(\frac{r_i(0)}{2}\right)e^{Ct}.$$
which implies that $r_i(t)$ can not go to $0$ in finite time by the lower bound estimate. However the above upper bound estimate is not useful, since $\tanh(x)\leq 1,$ for any $x>0.$ That is the reason why we need a priori upper bound estimate of the solutions in Theorem~\ref{thm:upperest}. By Theorem~\ref{thm:upperest}, there is some constant $C_3$ such that
$$r_i(t)\leq C_3,\quad \forall\ i\in \mathcal{T}_0,\ t\in[0,T).$$ Hence, by the extension theorem of solutions in ODE theory, the solution exists for all $t\geq 0,$ i.e. $T=\infty.$
\end{proof}

\subsection{Uniqueness of the extended flow}
We introduce the following change of variables: for any $1\leq i\leq N,$
$$w_i(r_i)=\int_0^{r_i}\frac{1}{\sqrt{\sinh s}}ds.$$ Note that $w_i(r_i)$ is increasing in $r_i.$ This gives us a diffeomorphism
\begin{eqnarray*}\pmb{w}(\br): &&\R^N_+\to (0,c_0)^N,\\ &&\br\mapsto \pmb{w}(\br):=(w_1(r_1),\cdots, w_N(r_N)),\end{eqnarray*} where $c_0=\int_0^{\infty}\frac{1}{\sqrt{\sinh s}}ds.$ The inverse map of $\pmb{w}(\br)$ is denoted by $\br(\pmb{w}).$
Note that the extended combinatorial Yamabe flow \eqref{extflow} can be written as
\begin{equation}\label{eq:gra1-1}r_i'=-\nabla_{r_i}\wt{\mathcal{S}} \sinh r_i.\end{equation}
We introduce a new functional on $(0,c_0)^N$
\begin{equation}\label{eq:hatS}\hat{\mathcal{S}}(\pmb{w}):=\wt{\mathcal{S}}(\br(\pmb{w})).\end{equation} 
Hence the equation \eqref{eq:gra1-1} is equivalent to the one in the $\pmb{w}$-coordinate
$$w_i'=-\nabla_{w_i}\hat{\mathcal{S}}.$$ This means that the extended combinatorial Yamabe flow \eqref{extflow} is equivalent to a negative gradient flow of the functional $\hat{\mathcal{S}}$ in the $\pmb{w}$-coordinate.
A function $f$ on a convex subset $W$ of $\R^n$ is called $\kappa$-convex, for some $\kappa\in \R$, if the function
$f(\pmb{x})-\frac{1}{2}\kappa|\pmb{x}|^2$ is convex on $W.$ A function $f$ on an open subset $V$ of $\R^n$ is called semi-convex, if for any point $\pmb{x}$ in $V$ there is a convex neighborhood $W(\pmb{x})$ such that $f$ is $\kappa(\pmb{x})$-convex on it. We will prove the uniqueness of the solutions to the extended flow \eqref{extflow} by the uniqueness of the negative gradient flow of some $C^1$ semi-convex functional on a subset of $\R^N.$



\begin{theorem}[Uniqueness]\label{thm:uniqueness} Given any initial data, the solution to the extended combinatorial Yamabe flow \eqref{extflow} is unique.
\end{theorem}
\begin{proof} It suffices to show that for any $\br(0)\in \R^N_+,$ any two solutions $\br_1(t)$ and $\br_2(t)$ to \eqref{extflow} with same initial data $\br(0)$ satisfy
\begin{equation}\label{eq:unieq1}\br_1(t)=\br_2(t),\quad \forall\ t\in[0,1].\end{equation} For $i=1,2,$ let $\pmb{w}_i(t):=\pmb{w}(\br_i(t))$ be the transformed solutions in the $\pmb{w}$-coordinate. Note that by Theorem~\ref{thm-yang-write}, $\{\br_1(t):t\in[0,1]\}$ and $\{\br_2(t):t\in[0,1]\}$ lie in a compact subset in $\R^N_+,$ and hence $\{\pmb{w}_1(t):t\in[0,1]\}$ and $\{\pmb{w}_2(t):t\in[0,1]\}$ lie in a compact convex subset $W$ in $(0,c_0)^N.$

We claim that there is a finite positive constant $\lambda$ such that for any $\pmb{w}_1,\pmb{w}_2\in W,$
$$(\nabla_{\pmb{w}} \hat{\mathcal{S}}(\pmb{w}_1)-\nabla_{\pmb{w}} \hat{\mathcal{S}}(\pmb{w}_2))\cdot(\pmb{w}_1-\pmb{w}_2)+\lambda|\pmb{w}_1-\pmb{w}_2|^2\geq 0,$$ where $\hat{\mathcal{S}}$ is defined in \eqref{eq:hatS}. That is, $\hat{\mathcal{S}}$ is $(-\lambda)$-convex on $W.$ For any $\epsilon>0,$ we define the $\epsilon$-mollifier of $\wt{\mathcal{S}}$ as
$$\wt{\mathcal{S}}_{\epsilon}:=\wt{\mathcal{S}}*\varphi_{\epsilon},\quad \mathrm{on}\ (\epsilon,\infty)^N,$$ where $*$ denotes the convolution, $\varphi_{\epsilon}(\br):=\frac{1}{\epsilon^N}\varphi(\frac{\br}{\epsilon})$ is the standard mollifier with
\[\varphi(\br):=\left\{\begin{array}{ll} Ce^{-\frac{1}{1-|\br|^2}},& |\br|<1,\\
0,&|\br|\geq1,\end{array}\right.\] and $C$ is chosen such that $\int \varphi=1.$
 Suppose that $\epsilon$ is sufficiently small such that $$\br(W)\subset (\epsilon,\infty)^N.$$ Since $\wt{\mathcal{S}}$ is a $C^1$-smooth convex functional, $\wt{\mathcal{S}}_\epsilon$ is $C^\infty$-smooth convex on $(\epsilon,\infty)^N,$ and
$$\wt{\mathcal{S}}_\epsilon\to \wt{\mathcal{S}} \quad \mathrm{in}\ C^1 \ \mathrm{on}\ \R^N_+, \quad\epsilon\to 0.$$ Moreover,
$$\nabla \wt{\mathcal{S}}_\epsilon=\nabla\wt{\mathcal{S}}*\varphi_\epsilon=\wt{K}*\varphi_\epsilon,\quad \mathrm{on}\ (\epsilon,\infty)^N.$$ Set $$ \hat{\mathcal{S}}_\epsilon(\pmb{w}):= \wt{\mathcal{S}}_\epsilon(\br(\pmb{w})).$$ Then by the chain rule,
$$\nabla_{\pmb w}\hat{\mathcal{S}}_\epsilon=\nabla_\br\wt{\mathcal{S}}_{\epsilon}\frac{\partial \br}{\partial \pmb{w}},$$
\begin{eqnarray*}\nabla^2_{w_i w_j}\hat{\mathcal{S}}_{\epsilon}&=&\frac{\partial^2 \wt{\mathcal{S}}_{\epsilon}}{\partial r_i\partial r_j}\frac{\partial r_i}{\partial w_i}\frac{\partial r_j}{\partial w_j}+\frac{\partial \wt{\mathcal{S}}_{\epsilon}}{\partial r_i}\frac{\partial^2 r_i}{\partial w_i \partial w_j}\\
&=&\frac{\partial^2 \wt{\mathcal{S}}_{\epsilon}}{\partial r_i\partial r_j}\frac{\partial r_i}{\partial w_i}\frac{\partial r_j}{\partial w_j}+\frac{1}{2}\wt{K}_i*\varphi_{\epsilon}\cosh r_i \delta_{ij},\end{eqnarray*} where $\delta_{ij}=1$ if $i=j,$ and $\delta_{ij}=0$ otherwise.
Note that on the compact subset $W,$ there is a constant $\lambda$ such that $$\frac{1}{2}|\wt{K}_i*\varphi_{\epsilon}\cosh r_i|\leq C\cosh r_i\leq \lambda.$$  Hence $$\nabla^2\hat{\mathcal{S}}_{\epsilon}\geq \frac{\partial \br}{\partial \pmb{w}}\nabla^2_\br\wt{\mathcal{S}}_{\epsilon}\frac{\partial \br}{\partial \pmb{w}}-\lambda I\geq -\lambda I,$$ where we have used the convexity of $\wt{\mathcal{S}}_\epsilon$ and
$I$ is the identity matrix. Hence $\hat{\mathcal{S}}_{\epsilon}$ is a smooth $(-\lambda)$-convex function on $W.$ Hence it is well-known that, for any $\pmb{w}_1,\pmb{w_2}\in W,$
$$(\nabla_{\pmb{w}} \hat{\mathcal{S}}_{\epsilon}(\pmb{w}_1)-\nabla_{\pmb{w}} \hat{\mathcal{S}}_{\epsilon}(\pmb{w}_2))\cdot(\pmb{w}_1-\pmb{w}_2)+\lambda|\pmb{w}_1-\pmb{w}_2|^2\geq 0.$$ By passing to the limit, $\epsilon\to 0,$ the claim follows.

Consider the function $h(t)=|\pmb{w}_1(t)-\pmb{w}_2(t)|^2.$
Then \begin{eqnarray*}h'(t)&=&-2(\pmb{w}_1(t)-\pmb{w}_2(t))\cdot (\nabla_{\pmb{w}}\hat{\mathcal{S}}(\pmb{w}_1(t))-\nabla_{\pmb{w}}\hat{\mathcal{S}}(\pmb{w}_2(t)))\\
&\leq& 2\lambda |\pmb{w}_1(t)-\pmb{w}_2(t)|^2=2\lambda h(t).\end{eqnarray*}
Hence $h(t)\leq h(0) e^{2\lambda t},$ for all $t\in [0,1].$ By $h(0)=0,$ $h(t)\equiv0$ on $[0,1].$ This yields \eqref{eq:unieq1} and the theorem follows.

\end{proof}

As a corollary, we prove that  the solution of extended combinatorial Yamabe flow \eqref{extflow} extends the solution to combinatorial Yamabe flow \eqref{hyperyflow}.
\begin{corollary} For an initial data $\br(0)\in \MT,$ we denote by $\br(t)$ ($\wt{\br}(t)$ resp.) be the solution to combinatorial Yamabe flow \eqref{hyperyflow} (extended combinatorial Yamabe flow \eqref{extflow} resp.). Then
$$\br(t)=\wt{\br}(t),\quad \forall\ t\in [0,T),$$ where $T$ is the maximal existence time of $\br(t).$
\end{corollary}
\begin{proof} Note that for any $\br\in \MT,$ $\wt{\pmb{K}}(\br)=\pmb{K}(\br).$ Hence, the corollary follows from the uniqueness of the solutions to the above flows, see Proposition~\ref{prop:combflow} and Theorem~\ref{thm:uniqueness}.
\end{proof}



\section{Proofs of main theorems}
\subsection{Triangulations with tetra-degrees at least 23}
In this subsection, we consider the triangulations of 3-manifolds whose tetra-degree at each vertex is at least 23. We prove lower bound estimates for the solutions to extended combinatorial Yamabe flow \eqref{extflow}.
\begin{theorem}\label{thm:lowerest} Let $(M,\mathcal{T})$ be a triangulated 3-manifold satisfying
$$d_i\geq 23,\quad \forall i\in \mathcal{T}_0.$$ Let $\br(t),$ $t\in[0,\infty)$, be a solution to the extended flow \eqref{extflow}. Then there exists a constant $C>0,$ depending on the initial data $\br(0),$ such that
$$r_i(t)\geq C,\quad \forall\ i\in \mathcal{T}_0, t\in [0,\infty).$$\end{theorem}
\begin{proof} Set $g(t):=\min_{m\in \mathcal{T}_0} r_m(t).$ Then $g(t)$ is a locally Lipschitz function and for a.e. $t\in (0,\infty)$, there exists $i\in \mathcal{T}_0$ depending on $t,$ such that
\begin{equation}\label{eq:eq5-1}g(t)=r_i(t),\quad \mathrm{and}\quad g'(t)=r_i'(t).\end{equation}
Let $\alpha_1^{\EE}(\pmb{1})=3\arccos{1/3}-\pi.$ Set $\epsilon_0:=\alpha_1^{\EE}(\pmb{1})-\frac{4\pi}{23}.$ We know that $\epsilon_0>0$ since $\frac{4\pi}{\alpha_1^{\EE}(\pmb{1})}\approx 22.80.$ Note that by Proposition~\ref{prop:regulartetra}, there exists a constant $C>0$ such that for any $s\leq C,$
$$\alpha_1(s\pmb{1})\geq \alpha_1^{\EE}(\pmb{1})-\epsilon_0.$$

We claim that \begin{equation*}\label{eq:eq5-2}g'(t)\geq 0,\quad \mathrm{for\ a.e.}\ t\in \{g<C\}.\end{equation*}
Let $t\in (0,\infty)$ and $i\in \mathcal{T}_0$ satisfying \eqref{eq:eq5-1}, and $t\in \{g<C\}.$ Then for any tetrahedron $\{ijkl\}$ incident to $i$ with the ball packing $\br(t),$ $$r_i(t)=\min\{r_i(t),r_j(t),r_k(t),r_l(t)\}<C.$$ By Theorem~\ref{thm:compare},
$$\wt{\alpha}_{ijkl}(\br(t))\geq \alpha_1(r_i(t)\pmb{1})\geq  \alpha_1^{\EE}(\pmb{1})-\epsilon_0.$$ Since $d_i\geq 23,$
$$\wt{K}_i(\br(t))=4\pi-\sum_{\{ijkl\}\in\mathcal{T}_3}\wt{\alpha}_{ijkl}(\br(t))\leq 4\pi-23(\alpha_1^{\EE}(\pmb{1})-\epsilon_0)=0.$$
This yields that by \eqref{extflow},
$$g'(t)=r_i'(t)=-\wt{K}_i \sinh r_i\geq 0.$$ This proves the claim. The theorem follows from the claim and Lemma~\ref{lem:calc}.

\end{proof}

\begin{lemma}\label{Lemma-ODE-asymptotic-stable}
(\cite{P1}) Let $V\subset \mathds{R}^n$ be an open set, $f\in C^1(V,\mathds{R}^n)$. Consider an autonomous ODE system
\begin{equation}\label{eq:5-1-2}\frac{d}{dt}{\pmb{x}(t)}=\pmb{f}(\pmb{x}(t)),~~~\pmb{x}(t)\in V.\end{equation}
Assuming $\pmb{x}^*\in V$ is a critical point of $f$, i.e. $\pmb{f}(\pmb{x}^*)=0$. If all the eigenvalues of the Jacobian matrix $\frac{\partial \pmb{f}}{\partial\pmb{x}}(\pmb{x}^*)$ have negative real part, then $\pmb{x}^*$ is an asymptotically stable point. More specifically, there exists a neighbourhood $\wt{V}\subset V$ of $\pmb{x}^*$, such that for any initial $\pmb{x}(0)\in \wt{V}$, the solution $\pmb{x}(t)$ to the equation \eqref{eq:5-1-2} exists for all time $t\in[0,\infty)$ and converges exponentially fast to $\pmb{x}^*$.
\end{lemma}

Now we are ready to prove one of our main results, Theorem~\ref{thm:23degree}.
\begin{proof}[Proof of Theorem~\ref{thm:23degree}] For any initial data $\br(0),$ let $\br(t)$ be a solution to the extended combinatorial Yamabe flow \eqref{extflow}. By Theorem~\ref{thm:lowerest} and Theorem~\ref{thm:upperest}, there are constants $C_1$ and $C_2$ such that
\begin{equation}\label{eq:eq5-2}C_1\leq r_i(t)\leq C_2,\quad \forall i\in \mathcal{T}_0, t\in [0,\infty).\end{equation}

We consider the first assertion, i.e. the existence of ball packing with vanishing extended combinatorial curvature.
For any $t>0,$ by Theorem~\ref{thm:extendedf},
\begin{eqnarray*}\frac{d}{dt}(\wt{\mathcal{S}}(\br(t)))&=&\sum_{i} \wt{K}_i(\br(t))r_i'(t)=-\sum_i \wt{K}^2_i(\br(t))\sinh(r_i(t))\\
&\leq&0.
\end{eqnarray*}
Hence $\wt{\mathcal{S}}(\br(t))$ is non-increasing. By \eqref{eq:eq5-2}, $\wt{\mathcal{S}}(\br(t))$ is bounded from below, which yields that there is a finite constant $C$ such that $$\lim_{t\to\infty}\wt{\mathcal{S}}(\br(t))=C.$$ Consider the sequence $\{\wt{\mathcal{S}}(\br(n))\}_{n=1}^\infty.$ By the mean value theorem, for any $n\geq 1$ there exists $t_n\in(n,n+1)$ such that
\begin{eqnarray}\label{eq:eq51-1}\wt{\mathcal{S}}(\br(n+1))-\wt{\mathcal{S}}(\br(n))&=&\frac{d}{dt}\Big|_{t=t_n}(\wt{\mathcal{S}}(\br(t)))\nonumber\\
&=&-\sum_i \wt{K}^2_i(\br(t_n))\sinh(r_i(t_n)).
\end{eqnarray} Note that $$\lim_{n\to\infty}\wt{\mathcal{S}}(\br(n+1))-\wt{\mathcal{S}}(\br(n))=C-C=0.$$
Hence by \eqref{eq:eq51-1},
$$\lim_{n\to \infty} \wt{K}_i(\br(t_n))=0,\quad\forall \ i\in \mathcal{T}_0.$$
By \eqref{eq:eq5-2}, we may extract a subsequence of $\br(t_n),$ still denoted by $\br(t_n)$ for simplicity, such that
$$\br(t_n)\to \overline{\br},\quad n\to\infty.$$
By the continuity of $\wt{K}_i,$ we have
$$\wt{K}_i(\overline{\br})=0,\quad \forall  \ i\in \mathcal{T}_0.$$ This proves the first assertion.

For the second assertion, let $\br^*$ be a real ball packing with vanishing combinatorial scaler curvature. By Theorem~\ref{thm:alternative}, $\br^*$ is the unique ball packing with such a property. By the proof above, there is a sequence $t_n\to\infty,$ such that \begin{equation}\label{eq:eqpart1}\br(t_n)\to \br^*.\end{equation} For any $\br\in\R^N_+,$ set
$$\pmb{f}(\br):=-\wt{K}_i(\br)\sinh(r_i).$$ Then the extended combinatorial Yamabe flow \eqref{extflow} can be written as \begin{equation}\label{eq:5-1-1}\frac{d}{dt}\br(t)=\pmb{f}(\br(t)).\end{equation} Note that critical points of $\pmb{f},$ corresponding to ball packings with vanishing combinatorial scaler curvature, consists of a single point $\br^*.$ We calculate the Jacobian matrix of the map $\pmb{f}$ at $\br^*,$
$$\frac{\partial \pmb{f}}{\partial \br}(\br^*)=-\Sigma \left(\frac{\partial \pmb K}{\partial \br}(\br^*)\right),$$ where $\Sigma=\mathrm{diag}\{\sinh r_1^*,\cdots,\sinh r_N^*\}.$ Hence
$$-\Sigma \left(\frac{\partial \pmb K}{\partial \br}(\br^*)\right)=-\Sigma^{\frac{1}{2}}\Sigma^{\frac{1}{2}}\left(\frac{\partial \pmb K}{\partial \br}(\br^*)\right) \Sigma^{\frac{1}{2}}\Sigma^{-\frac{1}{2}},$$ which has the same spectrum as $$-\Sigma^{\frac{1}{2}}\left(\frac{\partial \pmb K}{\partial \br}(\br^*)\right) \Sigma^{\frac{1}{2}}.$$ Note that
$\frac{\partial \pmb K}{\partial \br}(\br^*)$ is positive definite by Lemma~\ref{lem:rivin}. Hence
$\frac{\partial \pmb{f}}{\partial \br}(\br^*)$ is negative definite. By Lemma~\ref{Lemma-ODE-asymptotic-stable},
$\br^*$ is an asymptotically stable point of the flow \eqref{eq:5-1-1}, which is equivalent to \eqref{extflow}. By \eqref{eq:eqpart1}, the extended combinatorial Yamabe flow \eqref{extflow} converges to $\br^*$ for any initial data $\br(0).$ This proves the second assertion. The theorem follows.
\end{proof}

\subsection{Triangulations with tetra-degrees at most 22}
In this subsection, we consider the triangulations of 3-manifolds whose tetra-degree at each vertex is at most 22. For such a triangulation and any ball packing, we prove that there is a vertex has positive extended combinatorial scaler curvature.
\begin{theorem}\label{thm:positivecurvature} Let $(M,\mathcal{T})$ be a triangulated 3-manifold such that $d_i\leq 22$ for all $i\in\mathcal{T}_0.$ Then for any ball packing $\br\in R^N_+,$ there is a vertex $i,$ such that
$$\wt{K}_i\geq \epsilon_0,$$ where $\epsilon_0:=4\pi-22\alpha_1^\EE(\pmb{1})>0.$
\end{theorem}
\begin{proof} Let $i$ be the vertex such that $r_i=\max_{j\in \mathcal{T}_0}r_j.$ Then by Theorem~\ref{thm:secondcompare}, for any tetrahedron $\{ijkl\}$ incident to $i,$ $$\wt{\alpha}_{ijkl}\leq \alpha_1^\EE(\pmb{1}).$$ Since $d_i\leq 22,$
$$\wt{K}_i=4\pi-\sum_{\{ijkl\}\in \mathcal{T}_3}\wt{\alpha}_{ijkl}\geq 4\pi-22\alpha_1^\EE(\pmb{1}).$$ This proves the theorem.
\end{proof}

By the above result, we obtained a refined upper estimate for the solutions to the extended combinatorial Yamabe flow \eqref{extflow}.
\begin{theorem}\label{thm:estimate22} Let $(M,\mathcal{T})$ be a triangulated 3-manifold such that $d_i\leq 22$ for all $i\in\mathcal{T}_0.$ Let $\br(t)$ be a solution to the extended combinatorial Yamabe flow \eqref{extflow}. Set $r_M(t)=\max_{i\in\mathcal{T}_0}r_i(t).$ Then $r_M(t)$ is non-increasing. Moreover, for a.e. $t\in [0,\infty)$
$$r_M'(t)\leq -\epsilon_0 \sinh (r_M(t)), $$ where $\epsilon_0:=4\pi-22\alpha_1^\EE(\pmb{1})>0.$
\end{theorem}
\begin{proof} It suffices to prove the second assertion. It is well-known that for a.e. $t\in[0,\infty),$ there exists a vertex $i$ depending on $t$ such that
$$r_M(t)=r_i(t),\quad r_M'(t)=r_i'(t).$$ By the proof of Theorem~\ref{thm:positivecurvature},
$$\wt{K}_i(\br(t))\geq \epsilon_0.$$ Hence by \eqref{extflow},
$$r_M'(t)=r_i(t)\leq -\epsilon_0 \sinh (r_i(t))= -\epsilon_0 \sinh (r_M(t)).$$ This proves the theorem.
\end{proof}

Now we can prove Theorem~\ref{thm:less22}.
\begin{proof}[Proof of Theorem~\ref{thm:less22}]
The first assertion follows from Theorem~\ref{thm:positivecurvature}.

For the second assertion, let $\br(t)$ be a solution to the extended combinatorial Yamabe flow \eqref{extflow}.
By Theorem~\ref{thm:estimate22}, for a.e. $t\in [0,\infty),$
$$\left(\ln\tanh\frac{r_M(t)}{2}\right)'\leq -\epsilon_0,$$ where $r_M(t)=\max_{i\in\mathcal{T}_0}r_i(t).$  By integrating both sides from $0$ to $t,$
$$\tanh\frac{r_M(t)}{2}\leq  \tanh\frac{r_M(0)}{2}e^{-\epsilon_0 t}.$$ This proves the second assertion by passing to the limit, $t\to\infty.$ We prove the theorem.

\end{proof}

In the following, we prove Corollary~\ref{coro:regulartetra}.
\begin{proof}[Proof of Corollary~\ref{coro:regulartetra}] By Theorem~\ref{thm:23degree} and Theorem~\ref{thm:less22}, it suffices to prove that there is a real ball packing with vanishing combinatorial scaler curvature for $d\geq 23.$ For any $t>0,$ consider the ball packings $t\pmb{1}$ on the triangulation $\mathcal{T}.$  Since $\mathcal{T}$ is tetra-regular, for any vertex $i,$
$$K_i(t\pmb{1})=4\pi-d\cdot\alpha_1(t\pmb{1})=: g(t).$$ Note that by Proposition~\ref{prop:regulartetra} and $d\geq 23$,
$$\lim_{t\to 0}g(t)=4\pi-d\cdot\alpha_1^{\EE}(\pmb{1})<0,\ \mathrm{and}\ \lim_{t\to \infty}g(t)=4\pi.$$ By the strict monotonicity of $g(t),$ there is a unique $t_0\in(0,\infty)$ such that $K_i(t_0\pmb{1})=0$ for any vertex $i.$ Hence $t_0\pmb{1}$ is a real ball packing with vanishing combinatorial scaler curvature. By the rigidity result, Theorem~\ref{thm:alternative}, it is the unique ball packing with vanishing (extended) combinatorial scaler curvature. This proves the corollary.
\end{proof}

\section{Appendix}
We use the symbolic computations in Wolfram Mathematica~8. For any $\br=(r_1,r_2,r_3,r_4)\in \R^4_+,$ set $y_i=\coth r_i,$ for $1\leq i\leq 4,$ see Definition~\ref{def:rtoy}. For a tetrahedron $\tau=\{1234\}=\{ijkl\}$ with the ball packing $\br,$ we have the following results:
\begin{eqnarray*}
&&\frac{\partial \beta_{ij}}{\partial r_i}=\frac{\sinh^2 r_j\sinh r_k\sinh r_l}{2\sqrt{Q(\br)}\sinh (r_i+r_j+r_k)\sinh (r_i+r_j+r_l)}\times\\
&&\Bigg\{y_j^2 \Big[-y_k^2 - y_l^2 -
   2 \frac{y_i}{y_j} \Big(y_i y_k + y_i y_l + y_k y_l (2 + \frac{y_i}{y_j} )\Big) + (y_j - y_i) (2 y_i +
      y_k + y_l)\Big]\\
 &&-4 +2 y_j^2 + (y_k-y_l)^2 - 3 y_j (y_k + y_l) -
 3 y_i (2 y_j + y_k + y_l)\Bigg\},
\end{eqnarray*}

\begin{eqnarray*}
&&\frac{\partial \beta_{ij}}{\partial r_k}=\frac{\sinh (r_i+r_j)}{2\sqrt{Q(\br)}\sinh r_k\sinh (r_i+r_j+r_k)}[y_i+y_j+y_k-y_l],
\end{eqnarray*} and

\begin{eqnarray*}
&&\frac{\partial \alpha_{i}}{\partial r_i}=-\frac{\sinh r_i\sinh^2 r_j\sinh^2 r_k\sinh^2 r_l}{\sqrt{Q(\br)}\sinh (r_i+r_j+r_k)\sinh (r_i+r_j+r_l)\sinh (r_i+r_k+r_l)}\times\\
&&\Bigg\{y_i ^2 y_j y_k y_l \Big[2 y_i + y_j + y_k +
   y_l + \frac{y_i}{y_j} (y_i + y_k + y_l) + \frac{y_i}{y_k}(y_i + y_j + y_l) + \\&&\frac{y_i}{y_l} (y_i + y_j + y_k) + \Big(\frac{2}{y_i} + \frac{1}{y_j}+ \frac{1}{y_k} +
      \frac{1}{y_l}\Big) Q(\br)\Big]+6 - 2 y_k^2 + 6 y_k y_l\\&&-y_k^3 y_l - 2 y_l^2 + 2 y_k^2 y_l^2 - y_k y_l^3 -
 y_j^3 (y_k + y_l) + 4 y_i^2 (y_j + y_k + y_l)^2 +\\&&
 2 y_j^2 (-1 + y_k^2 + y_k y_l + y_l^2) -
 y_j (y_k^3 - 2 y_k^2 y_l + y_l (-6 + y_l^2) - 2 y_k (3 + y_l^2)) +\\&&
 y_i \Big(-2 y_j^3 + 10 y_k - 2 y_k^3 + 10 y_l + 3 y_k^2 y_l + 3 y_k y_l^2 -
    2 y_l^3 + 3 y_j^2 (y_k + y_l) + \\&&y_j (10 + 3 y_k^2 + 16 y_k y_l + 3 y_l^2)\Big)
 \Bigg\}.
\end{eqnarray*} Note that the terms in $[\cdots]$ also appear in the Euclidean setting by replacing $y_i$ by $\frac{1}{r_i},$ see \cite[Page~798]{G1}.

In particular, by evaluating at $\br=\pmb{1},$ we get
\begin{equation}\label{eq:eqa-1}\frac{\partial \pmb\alpha}{\partial \br}(\pmb{1})=c\left(\begin{matrix}
-3\cosh 2&1&1&1\\
1&-3\cosh 2&1&1\\
1&1&-3\cosh 2&1\\
1&1&1& -3\cosh 2
\end{matrix}\right),\end{equation} where $$c=\frac{2\sinh2}{(\cosh2-1)(2\cosh2+1)\sqrt{1+4\cosh2+3\cosh^22}}.$$ This yields that $\frac{\partial \pmb\alpha}{\partial \br}(\pmb{1})$ is negative definite.\\

\noindent
\textbf{Acknowledgements:} The first author is supported by the NSFC of China (No.11501027). Part of the work was done while the first author was visiting School of Mathematical Sciences, Fudan University in 2018. He would like to thank Professor Jiaxing Hong for many useful conversations during the visiting.


\bibliographystyle{spmpsci}
\bibliography{3dhyperbolic}

\bigskip

\bigskip
\end{document}